\documentclass[smallextended]{svjour3}
\usepackage{amsmath}
\usepackage{amssymb}
\usepackage[title]{appendix}
\usepackage[ruled,vlined]{algorithm2e}
\usepackage{pgfplots}
\newtheorem{Lemma}{Lemma}
\newtheorem{Theorem}[Lemma]{Theorem}

\newtheorem{Corollary}[Lemma]{Corollary}

\newtheorem{Definition}{Definition}
\newtheorem{Problem}{Optimization Problem}

\newenvironment{proofof}[1]{\noindent\emph{Proof of #1.}}{$\Box$}
\newenvironment{quiet-proof}{}{$\Box$}

\newcommand{\Indices}{\mathbb J}

\newcommand{\Natural}{\mathbb N}
\newcommand{\Real}{\mathbb R}

\newcommand{\Domain}{\mathcal D}
\newcommand{\RKHS}{\mathcal H}

\newcommand{\Hilbert}{\mathbb H}

\newcommand{\Sphere}{\mathbb S}

\newcommand{\Kernel}{\mathcal K}
\newcommand{\kernel}{K}

\newcommand{\One}{\mathbf 1}

\newcommand{\DA}{\operatorname{DA}}
\newcommand{\QMC}{\operatorname{QMC}}
\newcommand{\WW}{\operatorname{WW}}

\newcommand{\abs}[1]{\left| #1 \right|}
\newcommand{\norm}[1]{\left\Vert #1 \right\Vert}
\newcommand{\OrderH}{\operatorname{O}}

\newcommand{\To}{\rightarrow}

\newcommand{\argmin}{\operatorname{argmin}}
\newcommand{\argmax}{\operatorname{argmax}}
\newcommand{\linspan}{\operatorname{span}}

\newcommand{\Down}{\operatorname{\downarrow}}
\newcommand{\Eff}{r}

\newcommand{\relmiddle}[1]{\mathrel{}\middle#1\mathrel{}}

\begin{document}

\title{Sparse grid quadrature on products of spheres}
\journalname{Numerical Algorithms}

\author{Markus~Hegland \and Paul~Leopardi}
\institute{Mathematical Sciences Institute, Australian National University. \\ \email{paul.leopardi@gmail.com}}

\date{original: 29 January 2012, revised: 22 January 2015}

\maketitle

\begin{abstract}
\noindent
We examine sparse grid quadrature on weighted tensor products (\textsc{wtp}) of
reproducing kernel Hilbert spaces on products of the unit sphere $\Sphere^2$,
in the case of worst case quadrature error for rules with arbitrary quadrature weights.
We describe a dimension adaptive quadrature algorithm based on an algorithm of Hegland~\cite{Heg03}, 
and also formulate an adaptation of Wasil\-kowski and Wo\'znia\-kowski's \textsc{wtp} algorithm~\cite{WasW99}, 
here called the \textsc{ww} algorithm.
We prove that the dimension adaptive algorithm is optimal in the sense of Dantzig~\cite{Dan57}
and therefore no greater in cost than the \textsc{ww} algorithm.
Both algorithms therefore have the optimal asymptotic rate of convergence of quadrature error
given by Theorem~3 of Wasil\-kowski and Wo\'znia\-kowski~\cite{WasW99}.
A numerical example shows that, even though the asymptotic convergence rate is optimal,
if the dimension weights decay slowly enough, and the dimensionality of the problem is large enough,
the initial convergence of the dimension adaptive algorithm can be slow.
\end{abstract}

\section{Introduction}
This paper examines sparse grid quadrature on weighted tensor products of
reproducing kernel Hilbert spaces (\textsc{rkhs}) of real valued functions on the unit sphere $\Sphere^2 \subset \Real^3,$
in the case of worst case quadrature error for rules with arbitrary quadrature weights.
As per our previous paper on sparse grid quadrature on the torus \cite{HegL11},
the rates of convergence of the quadrature rules constructed here
are examined using the theory of Wasil\-kowski and Wo\'znia\-kowski~\cite{WasW99}.

The setting is the same as that used by Kuo and Sloan~\cite{KuoS05}, and Hesse, Kuo and Sloan~\cite{HesKS07}
to examine quasi-Monte Carlo (\textsc{qmc}) quadrature on products of the sphere $\Sphere^2,$
except that here we examine quadrature with arbitrary weights.

Quadrature on products of spheres is interesting, not just for purely theoretical reasons,
but also for practical reasons.
Integration over products of the unit sphere is equivalent to multiple integration over the unit sphere.
Such multiple integrals can be approximated in a number of ways, including Monte Carlo methods.
Applications of tensor product spaces on spheres and approximate integration over products of spheres
include quantum mechanics \cite{ZakHG03},
and transport and multiple scattering problems in various topic areas,
including acoustics \cite{Sat88}, optical scattering problems \cite{Alt88,KapLD01,StanO95}, 
and neutron transport problems \cite{Vin54}.
One prototypical problem to be solved is scattering by a sequence of spheres.
This can be modelled using a multiple integral of a function on the product of the spheres.
The decay in the weights of successive spheres could model the decreasing influence of scattering
on each successive sphere, 
as opposed to just cutting off the calculation after an arbitrary number of scatterings.

Quadrature with arbitrary weights, as opposed to equal weight quadrature, is interesting because, 
given the same set of quadrature points,
optimal quadrature weights give a quadrature error at least as good as that of equal weight quadrature.
This can result in the quadrature error converging more quickly to zero.
This is illustrated by a numerical example presented at the conference on Monte Carlo and Quasi-Monte Carlo methods in
Warsaw in 2010 \cite{HegL10}.
%

\begin{figure}[!ht]
\centering
\begin{tikzpicture}
\begin{loglogaxis}[
height=80mm,
width=80mm,
legend pos=outer north east,
xmin=1.0,
xmax=1.0e4,
ymin=1.0e-5,
ymax=1.0e-1,
xlabel={Cost (number of quadrature points, $n$)},
ylabel={Worst case quadrature error, $e$}
]
\pgfplotstableread{errhkswt-3-08-0.1.1-simplified.dat}\errhkswttable
\addplot[line width=1pt, color=green!70!black, mark=none] table[x=Points,y=HKS] {\errhkswttable};
\addplot[line width=1pt, color=red,   dashed,  mark=none] table[x=Points,y={HKS-optimal}] {\errhkswttable};
\addplot[line width=1pt, color=black, dotted,  mark=none] table[x=Points,y={Monte-Carlo}] {\errhkswttable};
\legend{HKS error,HKS optimal error,Monte Carlo rate}
\end{loglogaxis}
\end{tikzpicture}
\caption{Error of \textsc{hks} rule vs \textsc{hks} optimal rule for~$(\Sphere^2)^8$, $r=3$\,, $\gamma_{8,k}=0.1^k$.}
\label{HL-figure-0-1-HL-vs-HKS}
\end{figure}
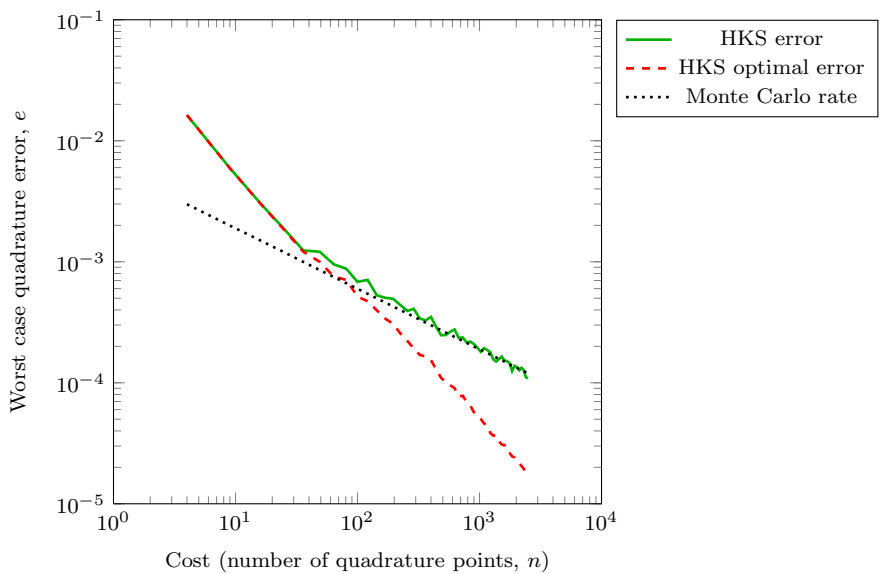

Figure~\ref{HL-figure-0-1-HL-vs-HKS}, based on
slide 25 of the presentation,
compares the performance of two quadrature rules, 
both based on the rule of Hesse, Kuo and Sloan \cite{HesKS07},
for a specific case of quadrature on $(\Sphere^2)^8$.
(The parameters $r$ and $\gamma$ mentioned in the caption are described in Section 2 below.)
The two rules use the same sequence of point sets, differing only in the quadrature weights.
The first rule, whose curve is labelled ``HKS error'', uses \textsc{qmc} weights.
Its error initially converges to zero rapidly, then begins to converge at the Monte Carlo error rate.
The second rule, labelled ``HKS optimal error'', uses optimal weights.
Its error continues to converge to zero rapidly.

As noted in our previous paper \cite{HegL11},
rates of convergence and criteria for strong tractability of quadrature with arbitrary weights
are known in the case of weighted Korobov spaces on the unit torus \cite{HicW01,SloW01}.
As far as we know, this paper is the first to examine the analogous questions 
for quadrature with arbitrary weights on the corresponding spaces on products of spheres.

The paper by Kuo and Sloan~\cite{KuoS05} gives the criteria for tractability and strong tractability
for \textsc{qmc} quadrature on products of spheres in the worst case setting.
Since, for a given finite set of points, quadrature with optimal weights yields an error
no larger than that of \textsc{qmc} quadrature, the quadrature problem for 
arbitrary weights in the worst case is tractable whenever the corresponding \textsc{qmc} quadrature 
problem is tractable.
The relevant definitions and criteria for tractability and strong tractability for
our setting are discussed in greater detail in Section 2.

The algorithms we examine are an adaptation of the algorithm used in
our previous paper \cite{Heg03,HegL11}, and an adaptation of the \textsc{wtp} algorithm of 
Wasil\-kowski and Wo\'znia\-kowski~\cite{WasW99}.
We examine these algorithms theoretically, 
giving bounds for the asymptotic convergence rate of quadrature error in the worst case.
We also examine the performance of these algorithms in practice, via a small number of numerical examples.

The main results of this paper are:
\begin{enumerate}
\item
(Theorem~\ref{HL-theorem-alg-dimadapt})

Under the conditions described in the theorem,
our dimension adaptive (\textsc{da})
algorithm is optimal in the sense of Dantzig \cite[Figure 3, p. 274]{Dan57} 
and therefore no greater in cost than the adapted
Wasil\-kowski and Wo\'znia\-kowski (\textsc{ww}) algorithm.
\item
(Theorem~\ref{HL-theorem-new-3}, Corollary~\ref{HL-cor-1} and Theorem~\ref{HL-theorem-new-4}).

The cost of the adapted \textsc{ww} algorithm in our setting is bounded such that,
whenever the problem is strongly tractable,
the asymptotic rate of convergence of the worst case quadrature error is 
essentially the same as that for the one dimensional problem.
\item
(Numerical examples).

Our numerical examples use exponentially decreasing sequences of dimension weights,
and a finite sequence of spherical designs as listed in Table~\ref{HL-table-DA-designs}.
This sequence of spherical designs yields a sequence of quadrature points 
on a single sphere having all of the properties needed by Theorem~\ref{HL-theorem-new-3}.

Since the problem is strongly tractable in this case,
if the sequence of spherical designs could be extended indefinitely, 
then the adapted \textsc{ww} algorithm would have an
asymptotic rate of convergence of the worst case quadrature error 
essentially the same as that for the one dimensional problem.
As a result of Theorem~\ref{HL-theorem-da-bounds}
the \textsc{da} algorithm would also have this same asymptotic convergence rate
in this case.

\item
(Figure~\ref{HL-figure-0-9-1-2-4-8-16}).

If the dimension weights decay slowly enough (e.g. $\gamma_{d,k} = 0.9^d$),
and the dimensionality is high enough (e.g. $d \geqslant 16$) 
the initial rate of convergence of the \textsc{da} algorithm can be slow.
Specifically, for  $\gamma_{d,k} = 0.9^d$ and $d = 16,$
the \textsc{da} algorithm needs more than $100\,000$ function evaluations to reduce 
the worst case quadrature error from 1 to 0.1.
\end{enumerate}

The remainder of this paper is organized as follows.
Section 2 describes the setting in detail,
and includes a discussion on the tractability of the problem.
Section 3 describes the optimization problem involved in dimension adaptive sparse grid quadrature.
Section 4 introduces the \textsc{da} algorithm and shows that it is optimal in the sense of Dantzig \cite{Dan57}.
Section 5 analyses a version of the \textsc{wtp} algorithm of Wasil\-kowski and Wo\'znia\-kowski,
derives bounds for the asymptotic rate of convergence of its worst case quadrature error to zero,
and applies these bounds to the \textsc{da} algorithm.
Section 6 contains numerical results, comparing implementations of the two algorithms, 
and showing how the \textsc{da} algorithm performs as the dimension is increased.
Appendix A contains a proof of Theorem~\ref{HL-theorem-new-3}.
\section{Setting}
\label{HL-section-Setting}

The setting used here is a special case of a general setting that
also applies to our previous paper \cite{HegL11}.

Let $\Domain \subset \Real^{s+1}$ be a compact manifold with probability measure $\mu.$ 
It follows that the constant function $\One,$ with $\One(x)=1$
for all $x\in \Domain,$ is integrable and $\int_\Domain \One(x)\, d\mu(x)=1.$ 
Then let $H$ be a Hilbert space of functions $f: \Domain \rightarrow \Real,$
with inner product $\langle \cdot,\cdot \rangle_H$, 
and kernel $\kernel,$ with the following properties.
\begin{enumerate}
  \item For every $x\in \Domain,$ the function $k_x\in H$, given by $k_x(y) := \kernel(x,y),$ satisfies
  \begin{equation}
  \label{HL-eq-repro}
  f(x) = \langle k_x, f \rangle_H, \quad \text{for all $f\in H$};
  \end{equation}
  \item Every $f\in H$ is integrable, and the constant function $\One$ is in $H$, such that
  \begin{equation}
  \label{HL-eq-integral}
  \int_\Domain f(x)\, d\mu(x) = \langle \One, f \rangle_H.
  \end{equation}
\end{enumerate}
We recognize $H$ as a reproducing kernel Hilbert space (\textsc{rkhs}).
In this framework, a quadrature rule $Q,$ defined by
\begin{align}
Q(f) &:= \sum_{i=1}^n w_i f(x_i)
\label{HL-eq-rule}
\end{align}
is a continuous linear functional and $Q(f) = \langle q, f\rangle_H$ with representer
\begin{align}
q &= \sum_{i=1}^n w_i k_{x_i},
\label{HL-eq-representer}
\end{align}
such that $Q(f) = \langle q, f\rangle_H.$
Here we have used~\eqref{HL-eq-repro} and the Riesz representation theorem for $H$.
In the remainder of this paper, we also refer to a function $q \in H$ of the form~\eqref{HL-eq-representer}
as a quadrature rule, with the understanding that $q$ represents the linear functional $Q$
of the form~\eqref{HL-eq-rule}.

Given the quadrature points $x_i,$ 
an optimal choice of weights $w_i$ minimizes the worst case quadrature error $e(q)$, 
which is 
\begin{align}
e(q) 
&:= \sup_{f \in H, \norm{f} \leqslant 1} \abs{ \langle \One, f \rangle_H - \langle q, f\rangle_H}
= \norm{\One - q}_H.
\label{HL-eq-e-def}
\end{align}
The optimal $q^*$ is thus defined as
\begin{align*}
q^* &:= \argmin_{q} \left\{\norm{\One - q}_H  \mid q \in \linspan\{ k_{x_1},\ldots,k_{x_n}\} \right\}.
\end{align*}
The weights of an optimal quadrature rule are thus obtained by solving a 
linear system of equations with a matrix whose elements are the values of the
reproducing kernel $\kernel(x_i,x_j)=\langle k_{x_i}, k_{x_j}\rangle_H.$ 
The right-hand side of these equations is a vector with elements all equal to one.

We now describe an auxiliary reproducing kernel Hilbert space $\RKHS$ of functions on $\Domain.$
The space $\RKHS$ has a kernel $\Kernel$ satisfying~\eqref{HL-eq-repro}, 
but instead of~\eqref{HL-eq-integral}, this auxiliary space satisfies
\begin{align*}
\int_\Domain f(x)\, d\mu(x) &= 0, \quad \text{for all $f\in \RKHS$}.
\end{align*}
Thus the function $\One$ is not an element of this space.
 
We now extend $\RKHS$ into the space $\RKHS^{\gamma},$
which consists of all functions of the form $g=a \One + f$, where $a \in \Real$,
and $f \in \RKHS,$ with the norm $\norm{\cdot}_{\RKHS^{\gamma}}$ defined by
\begin{align*}
\norm{g}_{\RKHS^{\gamma}}^2 &= |a|^2 + \frac{1}{\gamma}\norm{f}_\RKHS^2.
\end{align*}
It is easily verified that $\RKHS^{\gamma}$ is an \textsc{rkhs} with reproducing kernel
\begin{align*}
\Kernel_{\gamma}(x,y) &= 1 + \gamma \Kernel(x,y),
\end{align*}
where $\Kernel$ is the reproducing kernel of $\RKHS.$
In particular, the space $\RKHS^{\gamma}$ with kernel $\Kernel_{\gamma}$ 
contains the function $\One$ and satisfies both
properties~\eqref{HL-eq-repro} and~\eqref{HL-eq-integral}.
 
For functions on the domain $\Domain^d$ we consider the tensor product space
$\RKHS_d := \bigotimes_{k=1}^d \RKHS^{\gamma_{k}}$
where $1 \geqslant \gamma_1 \cdots \geqslant \gamma_d \geqslant 0.$ 
This is an \textsc{rkhs} of functions on $\Domain^d$ with reproducing kernel
$\Kernel_d(x,y) := \prod_{k=1}^d (1 + \gamma_k \,\Kernel(x_k,y_k))$
where $x_k,y_k\in\Domain$ are the components of $x,y\in \Domain^d.$ 
Also
\begin{align*}
\int_{\Domain^d} f(x)\,d\mu_d(x) &= \langle \One, f \rangle_{\RKHS_d},
\end{align*}
where $\mu_d$ is the product measure, $\langle\cdot,\cdot\rangle_{\RKHS_d}$
is the scalar product on the tensor product space $\RKHS_d,$ and $\One$ is the
constant function on $\Domain^d$ with value $1$. 

The specific setting for this paper is that of Kuo and Sloan~\cite{KuoS05}, with $s:=2,$
except that we allow quadrature rules with arbitrary weights.
We now describe this setting.
We take our domain $\Domain$ to be the unit sphere 
$\Sphere^2 := \{x \in \Real^3 \mid x_1^2 + x_2^2 + x_3^2 = 1 \}$,
and consider the real space $L_2(\Sphere^2)$ with respect to the uniform probability measure $\mu$
on $\Sphere^2$.
We use the orthonormal basis of real spherical harmonics
$Y_{\ell,m}(x),$ $\ell=0,\ldots,\infty,$ $m=1,\ldots,2\ell+1,$ 
as specified by Hesse, Kuo and Sloan \cite[Section 3.1]{HesKS07}.
As per \cite[(5)]{HesKS07}, the addition theorem for spherical harmonics with this normalization yields
\begin{align*}
\sum_{m=1}^{2\ell+1} Y_{\ell,m}(x)\ Y_{\ell,m}(y)
&=
(2 \ell + 1) P_{\ell}(x \cdot y)
\end{align*}
for all $x, y \in \Sphere^2$, where
$P_{\ell}$ is the Legendre polynomial of degree $\ell.$

For any function $f \in L_2(\Sphere^2),$ we expand $f$ in the Fourier series
\begin{align*}
f(x) &= \hat{f}_{0,0} + \sum_{\ell=1}^{\infty} \sum_{m=1}^{2\ell+1} \hat{f}_{\ell,m} Y_{\ell,m}(x).
\end{align*}

For a positive dimension weight $\gamma$, we define the \textsc{rkhs}
\begin{align*}
\Hilbert_{1,\gamma}^{(r)} 
&:= \{ f : \Sphere^2 \To \Real \mid \norm{f}_{\Hilbert_{1,\gamma}^{(r)}} < \infty \},
\intertext{where}
\langle f, g \rangle_{\Hilbert_{1,\gamma}^{(r)}} 
&:=
\hat{f}_{0,0}\, \hat{g}_{0,0} +
\gamma^{-1} \sum_{\ell=1}^{\infty} \sum_{m=1}^{2\ell+1} \big( \ell (\ell+1)\big)^r \,\hat{f}_{\ell,m}\, \hat{g}_{\ell,m}.
\end{align*}
Kuo and Sloan \cite{KuoS05} show that the reproducing kernel of $\Hilbert_{1,\gamma}^{(r)}$ is
\begin{align*}
\kernel_{1,\gamma}^{(r)}(x,y)
&:=
1 + \gamma A_r(x \cdot y), \quad \text{where for\ } z \in [-1,1],
\\
A_r(z) 
&:= 
\sum_{\ell=1}^{\infty} \frac{2\ell+1}{\big(\ell(\ell+1)\big)^{r}} P_{\ell} (z).
\end{align*}
The sum defining $A_r$ converges when $r > 3/2.$

For $\gamma := (\gamma_{d,1},\ldots,\gamma_{d,d}),$ we now define the tensor product space 
\begin{align*}
\Hilbert_{d,\gamma}^{(r)} &:= \bigotimes_{k=1}^d \Hilbert_{1,\gamma_{d,k}}^{(r)}.
\end{align*}
This is a weighted \textsc{rkhs} on $(\Sphere^2)^d,$ with reproducing kernel
\begin{align*}
\kernel_{d,\gamma}^{(r)}(x,y)
&:=
\prod_{k=1}^d \kernel_{1,\gamma_{d,k}}^{(r)}(x_k,y_k).
\end{align*}

Kuo and Sloan \cite{KuoS05} studied equal weight (\textsc{qmc}) quadrature on the space $\Hilbert_{d,\gamma}^{(r)}$,
and found that it is strongly tractable if and only if 
$\sum_{k=1}^d\gamma_{d,k} < \infty$ as $d \To \infty.$
The definition of strong tractability used by Kuo and Sloan \cite{KuoS05} is specific to \textsc{qmc} quadrature.
Here we expand the definition to quadrature with arbitrary weights.

We suppose for given $n > 0$
we can find points $\{x_1,\ldots,x_n\} \subset (\Sphere^2)^d$ and corresponding quadrature weights,
defining the quadrature rule $Q_{n,d}$ on $\Hilbert_{d,\gamma}^{(r)}$,
with worst case error $e_{n,d} := e(Q_{n,d}).$
We also define the quadrature rule $Q_{0,d} := 0$ so
that the corresponding worst case quadrature error is 
$e_{0,d} = \norm{\One}_{\Hilbert_{d,\gamma}^{(r)}} = 1$.

For $\varepsilon \in (0, 1)$ we want to find the smallest cost in terms of the number of function evaluations,
that is the smallest 
$n = \operatorname{cost}(d,\varepsilon)$ 
for which points $\{x_1, \ldots, x_n\} \subset (\Sphere^2)^d$ exist such
that $e_{n,d} \leqslant \varepsilon\, e_{0,d}.$ 
The integration problem for arbitrary quadrature weights in the worst-case setting is said to be \emph{strongly tractable} in the space 
$\Hilbert_{d,\gamma}^{(r)}$ if
\begin{align}
\label{eq-strong-tractability}
\operatorname{cost}(d,\varepsilon) &\leqslant C \varepsilon^{-p},
\end{align}
where $C$ and $p$ are non-negative constants independent of $\varepsilon$ and $d$. 
If~\eqref{eq-strong-tractability} holds then the infimum of $p$
is called the $\varepsilon$-exponent of strong tractability. 

Since, for a given finite set of points, quadrature with optimal weights yields an error
no larger than that of \textsc{qmc} quadrature, 
strong tractability for arbitrary weight quadrature rules holds for $\Hilbert_{d,\gamma}^{(r)}$ 
whenever strong \textsc{qmc} tractability holds.
In particular, as a consequence of a theorem of Kuo and Sloan \cite[Theorem 4]{KuoS05},
strong tractability for arbitrary weight quadrature rules holds for positive dimension weights
$\gamma_{d,k}$ when
\begin{align*}
\mathop{\limsup}_{d \To \infty} \sum_{k=1}^d \gamma_{d,k} &< \infty. 
\end{align*}
 
Hesse, Kuo and Sloan \cite{HesKS07} go on to construct sequences of \textsc{qmc} rules 
on the space $\Hilbert_{d,\gamma}^{(r)}$,
and prove that their worst case error converges at least as quickly as
the Monte Carlo error rate of $\OrderH(n^{-1/2}),$
where $n$ is the cost of the quadrature rule in terms of the number of points.

The work of Hickernell and Wo\'znia\-kowski~\cite{HicW01},
and Sloan and Wo\'znia\-kowski~\cite{SloW01},
on the weighted Korobov space of periodic functions on the unit cube,
and the work of Wasil\-kowski and Wo\'znia\-kowski~\cite{WasW99} 
on \textsc{wtp} quadrature on non-periodic functions on the unit cube,
combined with the observations above on strong tractability,
suggests bounds on the worst case error for the case of 
quadrature with arbitrary weights on the space $\Hilbert_{d,\gamma}^{(r)}$.
In the case of exponentially decreasing weights, as studied here,
one might expect that for $r > 3/2,$ given $\delta>0,$
there exist points on $(\Sphere^2)^d$ yielding an optimal weight quadrature rule whose
worst-case error would have an upper bound independent of $d,$ of order $\OrderH(n^{-r/2+\delta}).$
The specific form of this upper bound is suggested by the analysis of Sloan and Wo\'zniakowski
of \textsc{qmc} rules on the torus \cite[Theorem 3]{SloW01}.
The analysis in Sections 4, 5 and 6 below 
shows that the \textsc{da} algorithm satisfies the upper bound suggested here,
given the sequences of point sets on $\Sphere^2$ used in the numerical examples of Section 6.

\section{Optimization Problem}
\label{HL-section-Problem}

We first describe the optimization problem in the general \textsc{rkhs} setting, 
as given in Section~\ref{HL-section-Setting}.

Assume that a sequence of distinct quadrature points $x_1,x_2,\ldots \in \Domain,$
and a sequence of positive integers $n_0 < n_1 < \ldots$ are given
and are the same for all spaces $\RKHS^{\gamma}.$
The quadrature rules for $\RKHS^{\gamma}$
are then defined as some element of 
$V_j^{\gamma} := \linspan\{k^{\gamma}_{x_1},\ldots,k^{\gamma}_{x_{n_j}}\} \subset \RKHS^{\gamma}.$
Note that $V_j^{\gamma} \subset V_{j+1}^{\gamma}$ for all $j \geqslant 0.$
 
Denote the optimal rule in $V_j^{\gamma}$ by $q_j^{\gamma}.$ 
Now define the pair-wise orthogonal spaces
$U_j^{\gamma}$ by $U_0^{\gamma} := V_0^{\gamma},$ and by the orthogonal decomposition
$V_{j+1}^{\gamma} = V_j^{\gamma} \oplus U_{j+1}^{\gamma}.$
Using the fact that the $q_j^{\gamma}$ are optimal, it follows that
\begin{align*}
\delta_{j+1}^{\gamma} &:= q_{j+1}^{\gamma} - q_j^{\gamma}  \in U_{j+1}^{\gamma}
\end{align*}
and $\delta_0 := q_0^{\gamma} \in U_0^{\gamma} = V_0^{\gamma}.$ 
Note that,
while the dimension of $U_{j+1}^{\gamma}$ is $n_{j+1}-n_j,$ independently of $\gamma,$
\begin{align}
U_{j+1}^{\gamma} &\neq \linspan\{k_{x_{1+n_j}^{\gamma}},\ldots,k_{x_{n_{j+1}}^{\gamma}}\},
\label{HL-eq-fundamental-reason}
\end{align}
since the functions $k_{x_m}$ and $k_{x_n}$ for $m \neq n$ are, in general, not orthogonal to each other. 

We use the notation $\Indices:= \Natural^d$,
treating elements of $\Indices$ as indices,
with a partial order such that for $i, j \in \Indices$,
$i \leqslant j$ if and only if $i_h \leqslant j_h$ for all components.

For a index $i \in \Indices,$ let $\Down i$ denote the \emph{down-set} of $i,$
defined by \cite[p. 13]{DavP90}
\begin{align*}
\Down i &:= \{ j \in \Indices \mid j \leqslant i \}.
\end{align*}
Subsets of $\Indices$ are partially ordered by set inclusion.
For a subset $I \subset \Indices,$ let $\Down I$ denote the down-set of $I,$
defined by
$\Down I := \bigcup_{i \in I} \Down i.$
Then $\Down I$ is the smallest set $J \supseteq I$
such that if $i \in J$ and $j \leqslant i$ then $j \in J$.
Thus $\Down \Down I = \Down I.$

A sparse grid quadrature rule is then of the form
\begin{align*}
q \in V_I &:= \sum_{j\in I} \bigotimes_{k=1}^d V_{j_k}^{\gamma_{d,k}}
\end{align*}
for some index set $I.$ 
The orthogonal decomposition
$V_j^{\gamma} = \bigoplus_{i=1}^j U_i^{\gamma}$
and the observation~\eqref{HL-eq-fundamental-reason} 
yield the multidimensional orthogonal decomposition
\begin{align*}
V_I &= \bigoplus_{j\in \Down{I}} \bigotimes_{k=1}^d U_{j_k}^{\gamma_{d,k}}.
\end{align*}
A short derivation shows that an optimal $q\in V_I$ is given by
\begin{align*}
q_I^* = \sum_{j \in \Down{I}} \bigotimes_{k=1}^d \delta^{\gamma_{d,k}}_{j_k}.
\end{align*}
Thus both $V_I$ and $q_I^*$ are obtained in terms of the down-set $\Down{I},$
effectively restricting the choice of the set $I$ to index sets which are also down-sets.
(This construction is similar to the general sparse grid construction for functions on 
the $d$-dimensional torus, examined by Gerstner and Griebel \cite[Section 3.1]{GerG03}.)

This leads us to defining the concepts of an \emph{admissible} index set,
and an \emph{optimal} index set.
An admissible index set $I$ satisfies the \emph{admissibility condition}
(similar to that of Gerstner and Griebel \cite[Section 3.1]{GerG03})
\begin{align}
I &= \Down I.
\label{HL-eq-admiss}
\end{align}

An optimal index set is one which minimizes the error for a given cost, 
or minimizes the cost for a given error.
Here, the cost is the number of quadrature points, which is the dimension of $V_I$.
(This is a similar concept of optimality to that mentioned by Griebel and Knapek
in the context of approximation spaces \cite[Section 5]{GriK00}.)

We now make the definition of an optimal index set more precise.
We first define $\nu_{j_k} := \dim U_{j_k}^{\gamma_{d,k}}$ and 
$\delta_{j_k}^{(k)} := \delta_{j_k}^{\gamma_{d,k}}.$
For the remainder of this section, we use $\varepsilon \in (0,1)$ to denote the required upper bound
on quadrature error.
The optimization problem then uses the following definitions.
\begin{Definition}
\label{HL-def-nu-p}
For index $j \in \Indices$, define 
\begin{align*}
  \nu_j &:= \prod_{k=1}^d \nu_{j_k}, \quad
  \Delta_j := \bigotimes_{k=1}^d \delta_{j_k}^{(k)}, \quad
  p_j:=\norm{\Delta_j}^2, \quad 
  \Eff_j := p_j/\nu_j.
\end{align*}
For subset $I \subset \Indices$, define
\begin{align*}
\nu(I) &:= \sum_{j\in I} \nu_j, \quad
p(I) := \sum_{j\in I} p_j.
\intertext{Also, define $P := 1-\varepsilon^2.$}
\intertext{Here, $j_k$ is the $k$th component of the index $j.$}
\end{align*}
\end{Definition}
We use these definitions in the following sense.
The quadrature rule $q_I$ is given by a sum of incremental rules $\Delta_j,$
indexed by the multi-index $j$.
The ``profit'' $p_j$ for each incremental rule is its squared norm.
The cost $\nu_j$ of each incremental rule is the number of extra points 
the incremental rule contributes to 
the overall quadrature rule, assuming that the admissibility condition~\eqref{HL-eq-admiss} applies.
Thus the cost $\nu(I)$ of a quadrature rule, 
as given by the number of function evaluations,
is just the total cost of the incremental rules.
The ratio $\Eff_j = p_j/\nu_j$ is called the \emph{efficiency} of 
the incremental rule $\Delta_j.$

Due to the properties of $\nu_{j_k}$ and $\Delta_j,$
$\nu$ and $p$ satisfy
\begin{align}
\nu_j, \nu(I) \in \Natural_+, \quad
0 < p_j < 1, \quad
0 < p(I) < 1, \quad
p(\Indices) = 1.
\label{HL-eq-nu-p}
\end{align}

We now consider the following optimization problem,
posed as a minimization problem on the variable $I \subset \Indices$.
\begin{Problem}\label{HL-prob-0}
\begin{align*}
\text{Minimize} \quad &\nu(\Down I), \quad \text{subject to} \quad p(I) \geqslant P,
\end{align*}
for some $0 < P < 1$, where $\nu$ and $p$ satisfy~\eqref{HL-eq-nu-p}.
\end{Problem}
In other words, 
given a required upper bound $\varepsilon$ on the quadrature error,
the problem is to find the subset $I \subset \Indices$
with the smallest cost $\nu(\Down I) = \sum_{j\in \Down I} \nu_j$,
satisfying the constraint $1-\sum_{j\in I} p_j \leqslant \varepsilon^2$.

Optimization Problem~\ref{HL-prob-0} can have multiple solutions,
since for $H, I, J \subset \Indices$ if $J = \Down H = \Down I$
and both $p(H) \geqslant P$ and $p(I) \geqslant P$ then both
$H$ and $I$ are solutions to Optimization Problem~\ref{HL-prob-0}.
The following problem breaks this tie.

\begin{Problem}\label{HL-prob-alt-dcks}
\begin{align*}
\text{Maximize} \quad p(I) \quad \text{subject to}\ I\ \text{solving Optimization Problem~\ref{HL-prob-0}}.
\end{align*}
\end{Problem}

The solution of Optimization Problem~\ref{HL-prob-alt-dcks} satisfies the admissibility condition
\eqref{HL-eq-admiss}:
\begin{Lemma}\label{HL-lemma-dcks-adm}
If $I$ is a solution of Optimization Problem~\ref{HL-prob-alt-dcks}, then $I = \Down I.$
\end{Lemma}
\begin{proof}
  Let $J=\Down I$, where $I$ is a solution of Optimization Problem~\ref{HL-prob-alt-dcks}, 
  and therefore of Optimization Problem~\ref{HL-prob-0}.
  Then $I \subset J$ and thus $J$ satisfies the constraints of Optimization Problem~\ref{HL-prob-0}, since $p_i > 0$.
  Therefore $J$ is also a solution of Optimization Problem~\ref{HL-prob-0}, since $\nu(\Down J) = \nu(\Down I)$.  
  If $I \subsetneq J,$ it follows from
  $p_i>0$ that $p(J) > p(I),$ and so $I$ cannot be optimal. 
  Therefore $I = J$.
\qed\end{proof}

In view of the admissibility condition~\eqref{HL-eq-admiss}, 
we reformulate Optimization Problem~\ref{HL-prob-alt-dcks} as:
\begin{Problem}\label{HL-prob-dcks}
\begin{align*}
\text{Minimize} \quad &\nu(I), \quad 
\text{subject to} \quad I=\Down I, \quad \text{and} \quad p(I) \geqslant P,
\end{align*}
for some $0 < P < 1$, where $\nu$ and $p$ satisfy~\eqref{HL-eq-nu-p}.
\end{Problem}

As pointed out by (e.g.) Griebel and Knapek \cite{GriK00,GriK09},
some sparse grid problems can be formulated and solved as knapsack problems.
The resulting solution is optimal in terms of total profit for a given cost.

We call Optimization Problem~\ref{HL-prob-dcks} a \emph{down-set-constrained} binary knapsack problem. 
Each item in the knapsack is an incremental rule.
The relationships between Optimization Problem~\ref{HL-prob-dcks} and other more well-known knapsack problems
are described in more detail in Section~\ref{HL-section-Optimality}.

\section{Algorithm}
\label{HL-section-Optimality}
\subsection*{The dimension adaptive algorithm}
The dimension adaptive (\textsc{da}) algorithm to choose the set $I$
in the setting where $\Domain$ is $\Sphere^1$ is described in \cite{HegL11}.
The algorithm is quite general, and applies equally well to the current setting, where $\Domain$ is $\Sphere^2$.
We repeat the algorithm here as Algorithm~\ref{HL-alg-dimadapt},
with some changes in notation.
This is a greedy algorithm for Optimization Problem~\ref{HL-prob-dcks}.

The detailed description of Algorithm~\ref{HL-alg-dimadapt} uses the following notation.
Given the down-set $I \subset \Indices,$
we define $M(I)$ to be the set of minimal elements of $\Indices \setminus I$,
in other words,
\begin{align*}
M(I) &:= \big\{i \in \Indices \setminus I \mid I \cup \{i\} = \Down(I \cup \{i\}) \big\}.
\end{align*}

Given $j \in \Indices,$ define $S(j),$ 
the \emph{forward neighbourhood} of $j,$ \cite[p. 71]{GerG03} as
\begin{align*}
S(j) &:= 
\left\{i \in \Indices \mid j < i\ \text{and}\ (j \leqslant \ell < i \Rightarrow \ell=j) \right\},
\end{align*}
that is, $S(j)$ is the set of minimal elements of $\{i \in \Indices \mid j < i\}.$

\begin{algorithm}[!ht]
\KwData{error $\varepsilon$, incremental rules $\Delta_j$ and their costs
  $\nu_j$ for $j\in \Indices$}
\KwResult{$\varepsilon$ approximation $q$ and index set $I$}
$I := I_{(0)} := \{0\};\ $
$q := q_{(0)} := \Delta_0$\;
\While{$\norm{\One-q} > \varepsilon$}{
  $j^{(t+1)} := \argmax_i \{\Eff_i \mid i \in M(I_{(t)})\}$\;
  $I := I_{(t+1)} := I_{(t)} \cup \{j^{(t+1)}\};\ $
  $q := q_{(t+1)} := q + \Delta_{j^{(t+1)}}$ \;
}
\caption{The dimension adaptive (\textsc{da}) algorithm.}
\label{HL-alg-dimadapt}
\end{algorithm}


To show how the minimal set $M(I)$ can be effectively determined at each step of
Algorithm~\ref{HL-alg-dimadapt}, we use the subscript $(t)$ to keep track of the steps.

We treat Algorithm~\ref{HL-alg-dimadapt} as starting with 
the down-set $I_{(0)}:=\{0\},$
the index $j^{(0)}:=0$ and 
the minimal set $M_{(0)}:=\{0\}.$
At the start of step $t+1$ of the algorithm, 
we have previously computed the 
down-set $I_{(t)},$
the index $j^{(t)}$ and the 
minimal set $M_{(t)},$
such that 
$j^{(t)} \in M_{(t)},$
and (for $t>0$)
$M_{(t)} = M(I_{(t-1)})$
and $I_{(t)} = I_{(t-1)} \cup \{j^{(t)}\}.$
To construct $M_{(t+1)} = M(I_{(t)}),$
set 
\begin{align*}
S_{(t+1)} &:= \big\{i \in S(j^{(t)}) \mid I_{(t)} \cup \{i\} = \Down(I_{(t)} \cup \{i\})\big\}
\intertext{and}
M_{(t+1)} &:= \left(M_{(t)} \setminus \{j^{(t)}\}\right) \cup S_{(t+1)}.
\end{align*}
Note that
\begin{align*}
   \left(M_{(t)} \setminus \{j^{(t)}\}\right) \cup S_{(t+1)}
&= \left(M_{(t)} \cup S_{(t+1)}\right) \setminus \{j^{(t)}\}.
\end{align*}
As every minimal element of $\Indices \setminus I_{(t)}$ is either an element of $M_{(t)}$ 
(but not $j^{(t)}$) or an element of $S_{(t+1)}$ we see that $M_{(t+1)}$
is equal to $M(I_{(t)}).$

\begin{remark}
The details of the algorithm of Gerstner and Griebel \cite[Section 3.2]{GerG03} 
are similar to those of Algorithm~\ref{HL-alg-dimadapt}.
Their \emph{old index set} $\mathcal{O}$ corresponds to our down-set $I,$
and their \emph{active index set} $\mathcal{A}$ corresponds to our minimal set $M(I).$
\end{remark}

\subsection*{Complexity of Algorithm~\ref{HL-alg-dimadapt}}

Here we briefly examine the computational effort involved in determining the down-set $I_{(t)}$ using 
Algorithm~\ref{HL-alg-dimadapt}.
As noted above, Algorithm~\ref{HL-alg-dimadapt} is similar to 
the algorithm of Gerstner and Griebel.
Thus an analysis of complexity along the lines of their \cite{GerG03} Section 4.4 is appropriate.

At step $t+1$ of Algorithm~\ref{HL-alg-dimadapt}, the index $j^{(t)}$ is removed from 
the minimal set $M_{(t)},$
the backward neighbourhood of each of $d$ potential new indices is checked,
and up to $d$ new indices are inserted to form the updated set $M_{(t+1)}.$
Thus the set $M_{(t)}$ contains at most $t (d-1)$ elements.
(Given $i \in \Indices,$ the \emph{backward neighbourhood} $R(i)$ of $i$ 
is the set of maximal elements of $\{\ell \in \Indices \mid \ell < i\}$ \cite[p. 71]{GerG03}.)

The checking of the backward neighbourhood of each index $i \in S(j^{(t)})$
involves looking up each of the $d-1$ indices in $R(i) \setminus \{j\}$
to see if it is already contained in the index set $I_{(t)}.$
Assuming that the set $I_{(t)}$ is implemented using a hash table \cite{Gri98},
the effort involved in this lookup is independent of the size of $I_{(t)}$ 
and is therefore approximately $O(d^2).$

\begin{figure}[!ht]
\centering
\begin{tikzpicture}
\begin{loglogaxis}[
height=80mm,
width=80mm,
legend style={align=center},
legend pos=outer north east,
xmin=1.0,
xmax=1.0e3,
ymin=1.0,
ymax=1.0e5,
xlabel={Step number $t+1$ (size of $I_{(t)}$)},
ylabel={}
]
\addplot[line width=1pt, color=black, dotted,  mark=none] table[x=Indices,y=I_lookups] {serrwtpcount-3-04-0.5.dat};
\addplot[line width=1pt, color=brown!80!black, mark=x, mark size={2pt}] table[x=Indices,y=Minimals] {serrwtpcount-3-04-0.5.dat};
\addplot[line width=1pt, color=blue,           mark=*, mark size={1pt}] table[x=Indices,y=M_effort] {serrwtpcount-3-04-0.5.dat};
\addplot[line width=1pt, color=red,   dashed,  mark=none] table[x=Indices,y=M_bound] {serrwtpcount-3-04-0.5.dat};
\legend{Number of\\look ups of $I_{(t)}$,Size $\abs{M_{(t)}}$,Sum over $t$ of\\$d^2 \log_2(\abs{M_{(t)}})$,$t d^2\log_2(t d)$}
\end{loglogaxis}
\end{tikzpicture}
\caption{Effort involved in calculating \textsc{da} rules for~$(\Sphere^2)^4$, $r=3$\,, $\gamma_{4,k}=0.5^k$.}
\label{HL-figure-04-0-5-DA-effort}
\end{figure}
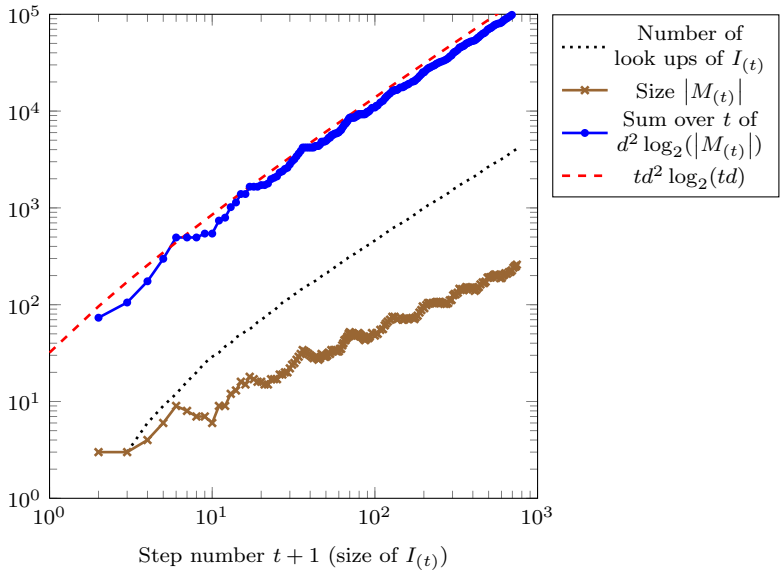

The index $j^{(t)}$ removed from $M_{(t)}$ is the index $i$ with 
the maximum value of the efficiency $\Eff_i.$
Let us assume that $M_{(t)}$ is implemented as an implicit heap,
or as a priority queue with the same time complexity for deletion of the largest element and for 
insertion of an element \cite{Jon86}.
To find $j^{(t)}$ and remove it from $M_{(t)}$ takes 
$O(d\log_2(\abs{M_{(t)}})) = O(d\log(t d))$ operations, 
such as comparison and index manipulation.
Each insertion also takes at most $O(d\log(t d))$ operations on $M_{(t)}$. 
Each insertion also requires the calculation of $\nu_i$ and $p_i$ for each index $i,$
each requiring effort $O(d),$ since each these are calculated via 
the product of one dimensional values.
At each step, the total effort required is therefore at most $O(d^2\log(t d)).$
The overall effort up to step $t+1$ is therefore at most $O(t d^2\log(t d)).$

Figure~\ref{HL-figure-04-0-5-DA-effort} shows some components of the cumulative effort involved in
calculating the \textsc{da} rules for the numerical example of
$(\Sphere^2)^4$, $r=3$,\, $\gamma_{4,k}=0.5^k,$
given the assumptions above.
Figure~\ref{HL-figure-0-5-DA-effort-2-4-8-16} compares the cumulative effort for $M_{(t)}$
for the \textsc{da} rules for the numerical example of
$(\Sphere^2)^d$, $r=3$,\, $\gamma_{d,k}=0.5^k,$ for $d=2,4,8,16,$
under the same assumptions.
See also Figure~\ref{HL-figure-0-5-HL-vs-WW} in Section~\ref{HL-section-Results},
and the detailed description of the numerical examples given in that Section.

In general, the number of new indices added to the minimal set $M_{(t)}$ depends 
not only on the dimension $d,$ but also on the dimension weights $\gamma_{d,k}.$
Thus the dimension weights determine the set $M_{(t)}$ at each step.
In particular, the size of $M_{(t)}$ may be much less than $t d.$
The dimension weights also determine the rate of convergence of the quadrature error $e(q_{(t)})$ to zero,
via  $\nu_{j^{(t)}}$ and $p_{j^{(t)}}$ for each index ${j^{(t)}},$
and hence determine the number of iterations $t$ needed before the error is reduced to $\varepsilon.$

A fully detailed analysis of the computational effort of Algorithm~\ref{HL-alg-dimadapt},
including an accurate estimate of $\abs{M_{(t)}}$ as a function of 
$d,$ $\gamma_{d,k},$ and $t,$ is thus potentially very complicated, 
and is beyond the scope of this paper.

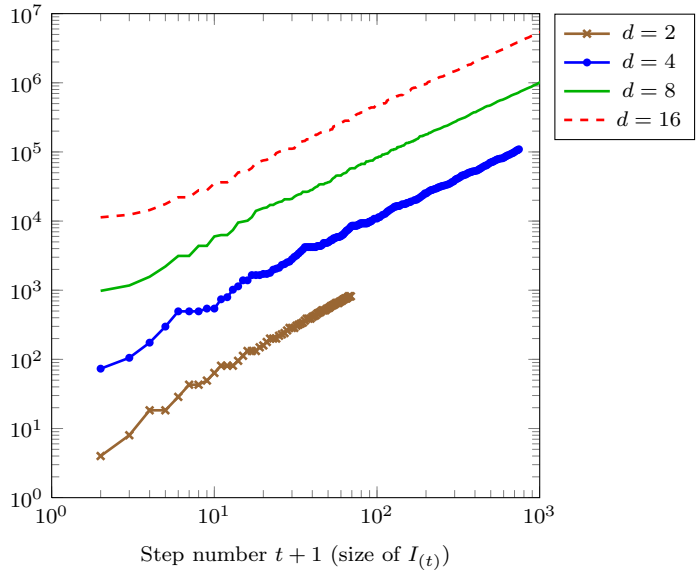
\begin{figure}[!ht]
\centering
\begin{tikzpicture}
\begin{loglogaxis}[
height=80mm,
width=80mm,
legend pos=outer north east,
xmin=1.0,
xmax=1.0e3,
ymin=1.0,
ymax=1.0e7,
xlabel={Step number $t+1$ (size of $I_{(t)}$)},
ylabel={}
]
\addplot[line width=1pt, color=brown!80!black, mark=x, mark size={2pt}] table[x=Indices,y=M_effort] {serrwtpcount-3-02-0.5.dat};
\addplot[line width=1pt, color=blue,           mark=*, mark size={1pt}] table[x=Indices,y=M_effort] {serrwtpcount-3-04-0.5.dat};
\addplot[line width=1pt, color=green!70!black, mark=none]               table[x=Indices,y=M_effort] {serrwtpcount-3-08-0.5.dat};
\addplot[line width=1pt, color=red,    dashed, mark=none]               table[x=Indices,y=M_effort] {serrwtpcount-3-16-0.5.dat};
\legend{$d=2$, $d=4$, $d=8$, $d=16$}
\end{loglogaxis}
\end{tikzpicture}
\caption{Sum over $t$ of $d^2 \log_2(\abs{M_{(t)}})$ for \textsc{da} rules for~$(\Sphere^2)^{d}$, $r=3$\,, $\gamma_{d,k}=0.5^k$.}
\label{HL-figure-0-5-DA-effort-2-4-8-16}
\end{figure}

\subsection*{Optimality of Algorithm~\ref{HL-alg-dimadapt}} 

Under certain conditions on $\nu$ and $p$, Algorithm~\ref{HL-alg-dimadapt}
solves Optimization Problem~\ref{HL-prob-dcks}. 
To see this, consider \emph{monotonicity} with respect to the lattice partial ordering of $\Indices$.
\begin{Definition}
\label{HL-def-monotonicity}
The function $p\in\Real_+^{\Indices}$ is \emph{monotonically decreasing} 
if $i < j$ implies that $p_i \geqslant p_j$.
If $i < j$ implies that $p_i > p_j$,
then $p\in\Real_+^{\Indices}$ is \emph{strictly decreasing}.
The definitions of ``monotonically increasing'' and 
``strictly increasing'' are similar.
\end{Definition}
Using Definition~\ref{HL-def-monotonicity}, the following theorem holds.
\begin{Theorem}\label{HL-theorem-alg-dimadapt}
  If $p\in\Real_+^{\Indices}$ is strictly decreasing and 
  $\nu\in\Natural_+^{\Indices}$ is monotonically increasing, 
  then Algorithm~\ref{HL-alg-dimadapt} yields a quadrature rule $q$ and index set $I$
  such that $I$ solves the down-set-constrained knapsack Optimization Problem~\ref{HL-prob-dcks},
  for $P=p(I)=\norm{\One-q}^2$.
\end{Theorem}

The proof of Theorem~\ref{HL-theorem-alg-dimadapt}
presented below proceeds in these stages.
\begin{enumerate}
\item
We introduce a related binary knapsack problem, 
and show that if $I$ is a solution of the binary knapsack problem,
and $I$ is also a down-set,
then $I$ is a solution of Optimization Problem~\ref{HL-prob-dcks}.
\item 
We describe a greedy algorithm for the binary knapsack problem (Algorithm~\ref{HL-alg-greedy-ks} below),
and show that if the efficiency $\Eff_j$ is strictly decreasing,
then each set $I$ produced by the greedy algorithm is a solution of the binary knapsack problem,
and is also a down-set, and therefore a solution of Optimization Problem~\ref{HL-prob-dcks}. 
\item
We show that if the efficiency is strictly decreasing, then Algorithm~\ref{HL-alg-greedy-ks} produces
the same sequence of sets as Algorithm~\ref{HL-alg-dimadapt}. 
\end{enumerate}

A binary (0/1) knapsack problem \cite{Dan57} related to Optimization Problem~\ref{HL-prob-dcks} is:
\begin{Problem}\label{HL-prob-ks}
\begin{align*}
\text{Minimize}
\quad
&\nu(I), 
\quad 
\text{subject to}
\quad
p(I) \geqslant P,
\end{align*}
for some $0 < P < 1$, where $\nu$ and $p$ satisfy~\eqref{HL-eq-nu-p}.
\end{Problem}
Usually a binary knapsack problem is posed as a maximization problem,
where the selection is from a finite set of items.
Here we have a minimization problem and a countably infinite set.
A finite minimization problem can always be posed as an equivalent maximization problem
\cite[p. 15]{MarT90}.
In the case of Problem~\ref{HL-prob-ks} this cannot be done, because the quantity to be maximized
(the sum of the costs of the elements not in the knapsack) would be infinite.
Instead, we must deal directly with the minimization form.

We now formulate a converse of Lemma~\ref{HL-lemma-dcks-adm}.
\begin{Lemma}\label{HL-lemma-ks-adm}
  If $I$ is a solution of the Optimization Problem~\ref{HL-prob-ks}, 
  and $I$ also satisfies the admissibility condition $I = \Down I,$ 
  then $I$ is a solution of Optimization Problem~\ref{HL-prob-dcks}.
\end{Lemma}
\begin{proof}
  If $I$ is a solution of the Optimization Problem~\ref{HL-prob-ks}, then
  $p(I) \geqslant P$. 
  If $I$ also satisfies the admissibility condition $I = \Down I,$ then
  $I$ satisfies the constraints of Optimization Problem~\ref{HL-prob-dcks} and thus
  minimizes $p$ under these constraints, i.e., is a solution of Optimization Problem~\ref{HL-prob-dcks}.
\qed\end{proof}
This justifies our calling Optimization Problem~\ref{HL-prob-dcks} a down-set-constrained knapsack problem.

If, in Optimization Problem~\ref{HL-prob-ks} 
we identify each set $I \subset {\Indices}$ with its \emph{indicator function}
${\mathcal I} \in \{0,1\}^{\Indices},$ where ${\mathcal I}_i = 1$ if and only if $i \in I,$
we obtain a more usual formulation of the binary knapsack problem:
\begin{Problem}\label{HL-prob-ks-0-1}

\begin{align*}
\text{Minimize} \quad &\sum_{i\in {\Indices}} \nu_i\,{\mathcal I}_i, \quad 
\text{subject to}
\quad
\sum_{i\in {\Indices}} p_i\,{\mathcal I}_i \geqslant P,
\quad
{\mathcal I} \in \{0,1\}^{\Indices},
\end{align*}
for some $0 < P < 1$, where $\nu$ and $p$ satisfy~\eqref{HL-eq-nu-p}.
\end{Problem}

Solving the binary knapsack problem is hard in general,
but for certain values of the constraint $P$, a greedy algorithm yields the solution.
These values are exactly the values for which the solution of
the binary knapsack problem equals the solution of the \emph{continuous} knapsack problem,
which uses the same objective function $\nu$ as Optimization Problem~\ref{HL-prob-ks-0-1}, and
relaxes the constraints ${\mathcal I}_i\in\{0,1\}$ to ${\mathcal I}_i\in[0,1].$ 
Dantzig \cite[Figure 3, p. 274]{Dan57} illustrates this for the classical binary knapsack problem
-- the finite maximization problem.
Martello and Toth \cite[Theorem 2.1, p. 16]{MarT90} 
give an explicit solution for the continuous problem, and a more formal proof.

The greedy algorithm for Optimization Problem~\ref{HL-prob-ks}
is based on the efficiency $\Eff_j = p_j/\nu_j.$ 
The algorithms generates the initial
values of an enumeration $j^{(t)}$ of $\Indices,$ $t \in \Natural_+,$ satisfying
   $\Eff_{j^{(t)}} \geqslant \Eff_{j^{(t+1)}}.$
The algorithm recursively generates $I_{(t)}$ from $I_{(t-1)},$ 
until for some $T$ the condition
   $p(I_{(T-1)}) < P \leqslant p(I_{(T)})$
holds, where
\begin{align*}
  I_{(t)} := \bigcup_{s=1}^t j^{(s)}.
\end{align*}

\begin{algorithm}[!ht]
\KwData{error $\varepsilon$, incremental rules $\Delta_j$ and their costs
  $\nu_j$ for $j\in \Indices$}
\KwResult{$\varepsilon$ approximation $q$ and index set $I$}
$I := I_{(0)} := \{0\};\ $
$q := q_{(0)} := \Delta_0$\;
\While{$\norm{\One-q} > \varepsilon$}{
  $j^{(t+1)} := \argmax_i \{\Eff_i \mid i \in \Indices \setminus I_{(t)}\}$\;
  $I := I_{(t+1)} := I_{(t)} \cup \{j^{(t+1)}\};\ $
  $q := q_{(t+1)} := q + \Delta_{j^{(t+1)}}$ \;
}
\caption{The greedy algorithm for Optimization Problem~\ref{HL-prob-ks}.}
\label{HL-alg-greedy-ks}
\end{algorithm}

This greedy algorithm, listed as Algorithm \ref{HL-alg-greedy-ks},
has the following properties.
\begin{Lemma}\label{HL-lemma-greedy-ks}
For any $0 < P < 1$,
Algorithm~\ref{HL-alg-greedy-ks} terminates for some $t=T.$

For each $t \geqslant 1$, the set generated by Algorithm~\ref{HL-alg-greedy-ks}, $I_{(t)}$
is the solution to Optimization Problem~\ref{HL-prob-ks} for $P=p(I_{(t)}).$
\end{Lemma}
\begin{proof}
  The algorithm terminates because $j^{(t)}$ is an enumeration of ${\Indices}$ 
  and therefore $\sum_{t=1}^\infty p_{j^{(t)}} = 1,$ since $\sum_{t=1}^\infty \Delta_{j^{(t)}} = \One,$
  but $P < 1.$

When $p(I_{(t)}) = P$ the constraints of Optimization Problem~\ref{HL-prob-ks} are satisfied.
Furthermore, as the method used
the largest $\Eff_i$, the objective function $\nu$ is minimised for Optimization Problem~\ref{HL-prob-ks}.
A more detailed proof can be constructed along the lines of the proof of Theorem 2.1 of Martello and Toth 
\cite{MarT90}.
\qed\end{proof}

The construction of the enumeration used in Algorithm~\ref{HL-alg-greedy-ks}
requires sorting an infinite sequence and is
thus not feasible in general,
but in the case where $p$ is strictly decreasing and $\nu$ is
monotonically increasing, the enumeration can be done recursively in finite time.
%

\begin{Lemma}\label{HL-lemma-greedy-down-set}
If $p$ is strictly decreasing and $\nu$ is monotonically increasing,
at each step $t>0$ of Algorithm~\ref{HL-alg-greedy-ks},
the index $j^{(t+1)}$ produced by the algorithm is a minimal element of the set 
$\Indices \setminus I_{(t)}$.
Also $j^{(0)} = 0.$
Therefore $I_{(t+1)}$ is a down-set.
\end{Lemma}
\begin{proof}
If $p$ is strictly decreasing and $\nu$ is monotonically increasing,
then $\Eff$ is monotonically decreasing.
By construction, $\Eff_{j^{(t)}} \geqslant \Eff_{j^{(t+1)}},$ 
so the enumeration must have $j^{(t)} < j^{(t+1)}$. 
It follows that $j^{(0)} = 0.$

For $t>0$, since $j^{(t+1)}$ is an enumeration of $\Indices,$ no element
occurs twice, and so $j^{(t+1)}\in \Indices \setminus I_{(t)}$. 
For $s>1$, any later element $j^{(t+s)}$ in the enumeration cannot be smaller than $j^{(t+1)},$
so $j^{(t+1)}$ is a minimal element of $\Indices \setminus I_{(t)},$
that is $j^{(t+1)} \in M_{(t+1)} = M(I_{(t)}).$
Since all elements smaller than $j^{(t+1)}$ occur earlier in the enumeration,
we must have $\Down j^{(t+1)} \subset I_{(t)} \cup \{j^{(t+1)}\}.$
Therefore, if $I_{(t)}$ is a down-set, then so is $I_{(t+1)}.$
Since $I_{(0)} = \{ 0 \},$ by induction, $I_{(t+1)}$ is always a down-set. 
\qed\end{proof}

\begin{Corollary}\label{HL-corr-greedy-solves-dcks}
For each $T \geqslant 1$, the set $I_{(T)}$ generated by Algorithm~\ref{HL-alg-greedy-ks}
is the solution to Optimization Problem~\ref{HL-prob-dcks} for $P=p(I_{(T)}).$
\end{Corollary}
\begin{proof}
This is an immediate consequence of Lemmas~\ref{HL-lemma-ks-adm},~\ref{HL-lemma-greedy-ks} and~\ref{HL-lemma-greedy-down-set}.
\qed\end{proof}


In the case where $\Eff$ is strictly decreasing, we have the following result.
\begin{Lemma}\label{HL-lemma-same}
If the efficiency $r$ is strictly decreasing, then Algorithm~\ref{HL-alg-greedy-ks} produces
the same sequence of sets  $I$ as Algorithm~\ref{HL-alg-dimadapt}. 
\end{Lemma}
\begin{proof}
From the proof of Lemma~\ref{HL-lemma-greedy-down-set}, we have that if $\Eff$ is strictly decreasing,
then for $t>0$, $j^{(t+1)}$ as per Algorithm~\ref{HL-alg-greedy-ks}
is always the element of $M_{(t+1)}$ which maximizes $\Eff.$
This is exactly $j := \argmax_i \{\Eff_i \mid i \in M(I_{(t)})\},$
as per Algorithm~\ref{HL-alg-dimadapt}.
For both algorithms, $I_{(0)} = \{0\}.$
\qed\end{proof}

All the pieces are now in place for the main proof of this Section.

~

\begin{proofof}{Theorem~\ref{HL-theorem-alg-dimadapt}}
From Corollary~\ref{HL-corr-greedy-solves-dcks}
we see that if the efficiency $\Eff$  is strictly decreasing,
then each set $I$ produced by the greedy algorithm (Algorithm~\ref{HL-alg-greedy-ks}) is a solution of 
Optimization Problem~\ref{HL-prob-dcks}. 
From Lemma~\ref{HL-lemma-same}
we see that if the efficiency $\Eff$ is strictly decreasing, then Algorithm~\ref{HL-alg-greedy-ks} produces
the same sequence of sets $I$ as Algorithm~\ref{HL-alg-dimadapt}. 

If $p$ is strictly decreasing and $\nu$ is monotonically increasing, 
then since $\Eff_j = p_j/\nu_j$ then $\Eff$ is strictly decreasing.
Therefore each set $I$ in the sequence produced by Algorithm~\ref{HL-alg-dimadapt} is a solution of
Optimization Problem~\ref{HL-prob-dcks}. 
\end{proofof}

\section{Error bounds}
\label{HL-section-Bounds}

We will now describe a second variant of \textsc{wtp} quadrature, 
$q^{(\WW)}$ on $\Hilbert_{d,\gamma}^{(r)},$ 
based on the same set of incremental rules as that used for
the sequence of quadrature rules $q^{(\DA)}$ 
described in Section~\ref{HL-section-Problem} above, 
but where the order in which the incremental rules are added to this variant is based on
the construction of Wasil\-kowski and Wo\'znia\-kowski~\cite[Section 5]{WasW99}.
The construction of this variant uses criteria similar to those used by
Wasil\-kowski and Wo{\'z}nia\-kowski \cite[Theorem 3]{WasW99}, but adapted to our setting.
These criteria are
\begin{align}
\norm{q_0}_{\Hilbert_{1,\gamma}^{(r)}} &\leqslant 1, \quad
\norm{q_j-q_{j-1}}_{\Hilbert_{1,\gamma}^{(r)}} \leqslant \sqrt{\gamma} C D^j \quad \text{for all}\ j \geqslant 1,
\ \text{and all}\ 0 < \gamma \leqslant 1
\label{HL-eq-new-39}
\intertext{(corresponding to Wasil\-kowski and Wo{\'z}nia\-kowski \cite[(39)]{WasW99}), and}
\nu_0 &= 1, \quad
\nu_j\,D^{j \rho} \leqslant 1 \quad \text{for all}\ j \geqslant 1
\label{HL-eq-new-36}
\intertext{(corresponding to Wasil\-kowski and Wo{\'z}nia\-kowski \cite[(36)]{WasW99}),
\newline
for some $D \in (0,1)$ and some positive $C$ and $\rho$.}
\notag
\end{align}

As a consequence of~\eqref{HL-eq-new-39}, we have 
\begin{align}
\norm{\Delta_j}_{\Hilbert_{d,\gamma}^{(r)}} 
&=
\prod_{k=1}^d \norm{\delta^{(k)}_{j_k}}_{\Hilbert_{1,\gamma_{d,k}}^{(r)}}
\leqslant b(d,j), \quad \text{where} 
\notag
\\
b(d,j) &:= \prod_{k=1}^d \left( \sqrt{\gamma_{d,k}}\ C D^{j_k} \right)^{1-\partial_{0,j_k}}.
\label{HL-eq-b-d-j-def}
\end{align}
Here, and below, we use 
\begin{align*}
\partial_{a,b}
&:=
\begin{cases}
0, & a \neq b
\\
1, & a = b
\end{cases}
\end{align*} 
to denote the Kronecker delta.

Let $(\xi_{d,k}),$ $k=1,\ldots,d,$ be a sequence of positive numbers. 
In contrast to Wasil\-kowski and Wo\'znia\-kowski~\cite[Section 5]{WasW99}, 
we do not stipulate that $\xi_{d,1}=1,$ but instead, $\xi_{d,1} \geqslant \sqrt{1-D^2}.$
Define
\begin{align}
\xi(d,j) &:= \prod_{k=1}^d \xi_{d,k}^{1-\partial_{0,j_k}}.
\label{HL-eq-xi-d-j-def} 
\end{align}

We therefore have $b(d,j)/\xi(d,j) \To 0$ as $\norm{j}_1 \To \infty.$
We order the incremental rules in order of non-decreasing $b(d,j)/\xi(d,j)$ for each index $j$,
creating an order on the indices $j^{(\WW)}(h)$ .
We adjust $\xi(d,k)$ so that this order agrees with the lattice partial ordering of the indices.
We now define $I_N^{(\WW)} := \{j^{(\WW)}(1),\ldots,j^{(\WW)}(N)\},$
and define the quadrature rule
\begin{align}
q_N^{(\WW)} &:= \sum_{j \in I_N^{(\WW)}} \Delta_j.
\label{HL-def-q-N-WW}
\end{align}

To obtain a quadrature error of at most $\varepsilon \in (0,1)$, we set
\begin{align}
N(d,\varepsilon,\eta) &:= 
\left|\left\{j \mid b(d,j)/\xi(d,j) > \big( \varepsilon/C_1(d,\eta) \big)^{1/(1-\eta)}\right\}\right|,
\label{HL-eq-N-d-eps-eta-def}
\intertext{where $\eta \in (0,1)$ and}
C_1(d,\eta) 
&:= 
\sqrt{\frac{\xi_{d,1}^{2(1-\eta)}}{1-D^2} \prod_{k=2}^d \left(1 + (C^2 \gamma_{d,k})^{\eta}\ \xi_{d,k}^{2(1-\eta)} \frac{D^{2\eta}}{1-D^{2\eta}}\right)}.
\label{HL-eq-C-1-d-eta-def}
\end{align}
Finally, we define
\begin{align}
q^{(\WW)}_{d,\varepsilon,\eta} &:= 
\sum_{j \in I_{N(d,\varepsilon,\eta)}}^{(\WW)} \Delta_j.
\label{HL-eq-new-45}
\end{align}

We can now present our version of Wasil\-kowski and Wo\'znia\-kowski's main theorem on 
the error and cost of \textsc{wtp} quadrature ~\cite[Theorem~3]{WasW99}.

\begin{Theorem}
\label{HL-theorem-new-3}
Let $\eta \in (0,1).$ 
Assume that a sequence of quadrature points $x_1,x_2,\ldots \in \Sphere^2,$
and a sequence of positive integers $1=n_0 < n_1 < \ldots$ are given
such that the corresponding optimal weight quadrature rules $q_j := q_j^1 \in \Hilbert_{1,1}^{(r)}$
satisfy~\eqref{HL-eq-new-39} and~\eqref{HL-eq-new-36}
for some $D \in (0,1)$ and some positive $C$ and $\rho$.
Then the quadrature rule $q^{(\WW)}_{d,\varepsilon,\eta}$ defined by~\eqref{HL-eq-new-45} has 
worst-case quadrature error $e(q^{(\WW)}_{d,\varepsilon,\eta}) \leqslant \varepsilon,$
and its cost (in number of quadrature points) is bounded by
\begin{align}
\operatorname{cost}&(q^{(\WW)}_{d,\varepsilon,\eta}) \leqslant C(d,\varepsilon,\eta)\left(\frac{1}{\varepsilon}\right)^{\rho/(1-\eta)},
\label{HL-eq-WW-cost-bound}
\end{align}
where
\begin{align}
C(d,\varepsilon,\eta) 
&:=
\frac{
 \operatorname{max}(\sqrt{\gamma_{d,1}}, \xi_{d,1})^{\rho}
 {\displaystyle \prod_{k=2}^d}
 \left(
  1 +
  \frac{\displaystyle C^{\rho} \gamma_{d,k}^{\rho/2}}{\displaystyle \xi_{d,k}^{\rho}}\ g(k,\varepsilon,\eta)
 \right)
 f_{k,\eta}^{\rho}
}{
 (1 - D^{\rho})(1-D^2)^{\rho/(2-2\eta)}
}
\label{HL-eq-C-d-def}
\\
f_{k,\eta}
&:= 
\left( 
 1+C^{2\eta}\gamma_{d,k}^{\eta}\xi_{d,k}^{2(1-\eta)}\frac{D^{2\eta}}{1-D^{2\eta}} 
\right)^{1/(2(1-\eta))},
\label{HL-eq-f-def}
\end{align}
\begin{align}
g(k,\varepsilon,\eta) &:= 
\left\lfloor
 \frac{ \log\left( 
  \frac{\displaystyle C \gamma_{d,k}^{1/2}}{\displaystyle (1-D^2)^{1/(2-2\eta)}}\ 
  \frac{\displaystyle \xi_{d,1}}{\displaystyle \xi_{d,k}}
  \left( 
   \prod_{i=2}^k f_{i,\eta} 
  \right) 
  \left(\frac{\displaystyle 1}{\displaystyle \varepsilon}\right)^{1/(1-\eta)}
 \right) }{\log \big(D^{-1}\big)} 
\right\rfloor_{+}.
\label{HL-eq-g-def}
\end{align}
By $\lfloor x \rfloor_{+}$, we mean $\max(0,x).$
\end{Theorem}

Wasil\-kowski and Wo\'znia\-kowski's proof, with $s:=2, \alpha:=1,$ 
applies directly to our Theorem~\ref{HL-theorem-new-3}, once the change in $\xi_{d,1}$ is taken into account.
For details of the proof, see Appendix A.

Corollary 1 of Wasil\-kowski and Wo\'znia\-kowski \cite[p. 434]{WasW99} presents a simpler bound for
the cost of their \textsc{wtp} algorithm, and their simplification also applies here.
\begin{Corollary}\label{HL-cor-1}
For every positive $\delta$ there exists a positive $c(d,\delta)$ such that the cost of
 the quadrature rule $q^{(\WW)}_{d,\varepsilon,\eta}$ defined by~\eqref{HL-eq-new-45} is bounded by
\begin{align*}
\operatorname{cost}(q^{(\WW)}_{d,\varepsilon,\eta}) &\leqslant c(d,\delta)\left(\frac{1}{\varepsilon}\right)^{\rho+\delta}.
\end{align*}
\end{Corollary}

Following Theorem 4 of Wasil\-kowski and Wo\'znia\-kowski~\cite{WasW99} we now examine the dependence
of our cost bound~\eqref{HL-eq-WW-cost-bound} on the dimension $d$.
If $\xi_{d,k}=c$ for all $k\geqslant 2$, for some positive constant $c$,
then for exponentially decreasing dimension weights $\gamma_{d,k} = g^k$ with $g < 1,$
we have
\begin{align*}
\sup_d \left\{ 
\sum_{k=2}^d \left( \frac{\gamma_{d,k}^{1/2}}{\xi_{d,k}} \right)^{\rho},\,
\sum_{k=2}^d \left( \xi_{d,k}^s \frac{\gamma_{d,k}^{1/2}}{\xi_{d,k}} \right)^{s\eta} \right\}
&=
\max \left\{ 
c^{-\rho} \frac{g^{\rho}}{1-g^{\rho/2}},\,
c^{s-s\eta} \frac{g^{s\eta}}{1-g^{s\eta/2}} \right\}
\\
&< \infty. 
\end{align*}

Examining Theorem 4 of of Wasil\-kowski and Wo\'znia\-kowski~\cite{WasW99}
and its proof in some detail, 
we see that a similar argument also applies to the $q^{(\WW)}$ rules.
This implies the following.
\begin{Theorem}\label{HL-theorem-new-4}
If $\xi_{d,k}=c$ for all $k\geqslant 2$, for some positive constant $c$,
then for exponentially decreasing dimension weights $\gamma_{d,k} = g^k$ with $g < 1,$
the algorithm giving the quadrature rule $q^{(\WW)}_{d,\varepsilon,\eta}$ defined by~\eqref{HL-eq-new-45}
is strongly polynomial, and in particular,
for each positive $\delta$ there exists $c_{\delta}$ such that
\begin{align*}
\operatorname{cost}(q^{(\WW)}_{d,\varepsilon,\eta}) 
&\leqslant c_{\delta}\left(\frac{1}{\varepsilon}\right)^{\rho/(1-\eta)+\delta},
\end{align*}
for all $\varepsilon \in (0,1)$ and all $d \geqslant 1.$
\end{Theorem}

The sequence of rules $q^{(\DA)}$ is at least as efficient as $q^{(\WW)},$
in the sense that $q^{(\DA)}$ is based on the optimal solution of 
the corresponding down-set-con\-strained continuous knapsack problem,
as explained in Section~\ref{HL-section-Optimality}.
As a direct consequence of Theorem~\ref{HL-theorem-new-3} and Corollary~\ref{HL-cor-1},
we therefore have the following result.
\begin{Theorem}
\label{HL-theorem-da-bounds}
Let $\eta \in (0,1).$
Assume that a sequence of quadrature points $x_1,x_2,\ldots \in \Sphere^2,$
and a sequence of positive integers $n_0 < n_1 < \ldots$ are given
such that the corresponding optimal weight quadrature rules $q_j := q_j^1 \in \Hilbert_{1,1}^{(r)}$
satisfy~\eqref{HL-eq-new-39} and~\eqref{HL-eq-new-36}
for some $D \in (0,1)$ and some positive $C$ and $\rho$.
Let $I_{(t)}, q_{(t)}^{(\DA)}$ be the index set and corresponding quadrature rule
generated by iteration $t$ of Algorithm~\ref{HL-alg-dimadapt}, based on the rules $q_j,$
for sufficiently small error $\varepsilon=\varepsilon_0$.

Then the quadrature rule $q_{(t)}^{(\DA)}$ has 
worst-case quadrature error $e(q_{(t)}^{(\DA)}) = \varepsilon_{(t)} := \sqrt{1-p(I_{(t)})},$
and its cost $\nu(I_{(t)})$ is bounded by
\begin{align*}
\nu(I_{(t)}) \leqslant C(d,\varepsilon_{(t)},\eta)\left(\frac{1}{\varepsilon_{(t)}}\right)^{\rho/(1-\eta)},
\end{align*}
where $C(d,\varepsilon_{(t)},\eta)$ is defined as per Theorem~\ref{HL-theorem-new-3}.
As a consequence,
for every positive $\delta$ there exists a positive $c(d,\delta)$ such that the cost of
 the quadrature rule $q_{(t)}^{(\DA)}$ is bounded by
\begin{align*}
\nu(I_{(t)}) &\leqslant c(d,\delta)\left(\frac{1}{\varepsilon_{(t)}}\right)^{\rho+\delta}.
\end{align*}

\end{Theorem}
\section{Numerical results}
\label{HL-section-Results}

With the estimates given by our analysis in hand, we now  compare these to 
the results of a selection of numerical examples.

Since the underlying domain $\Domain$ is $\Sphere^2$ rather than $\Sphere^1,$
we need to change some of the details of the algorithm in comparison to the algorithm used for the torus \cite{HegL11}.
Specifically, we need a sequence of rules on a single sphere,
which yields ``good enough'' worst case quadrature error with optimal weights.
The numerical examples of this Section use a sequence of point sets
consisting of unions of spherical designs with increasing numbers of points
and non-decreasing strengths.

For the unit sphere $\Sphere^2,$ 
a spherical design \cite{DelGS77} of strength $t$ and cardinality $m$ is a set of $m$ points
$X=\{x_1,\ldots,x_m\} \subset \Sphere^2$ such that the equal weight quadrature rule
\begin{align*}
Q_X(p) &:= \frac{1}{m} \sum_{h=1}^m p(x_h) 
\end{align*}
is exact for all spherical polynomials $p$ of total degree at most $t$.

For the numerical examples of this Section, a combination of (approximate) extremal (\emph{E}) 
and low cardinality (\emph{L}) spherical designs are used, 
according to Table~\ref{HL-table-DA-designs}.
These approximate spherical designs were all provided by Womersley \cite{CheW06,Wom09}.

\begin{table}[!ht]
\begin{center}
\begin{align*}
\begin{array}{|c|c|c|c|c|c|c|c|c|c|c|c|c|}
\hline
\text{Index}\ j       & 0 & 1 & 2 & 3 &  4 &  5 &  6 &   7 &   8 &   9 &   10 &   11 \\
\text{Type}           & L & L & E & L &  E &  L &  E &   L &   E &   L &    E &    L \\
\text{Strength}\ t      & 0 & 1 & 1 & 3 &  3 &  7 &  7 &  15 &  15 &  31 &   31 &   63 \\
\text{Cardinality}\ m & 1 & 2 & 4 & 8 & 16 & 32 & 64 & 129 & 256 & 513 & 1024 & 2049 \\
\hline
\end{array}
\end{align*}
\end{center}
\caption{Strength and cardinality of approximate spherical designs used with 
Algorithm~\ref{HL-alg-dimadapt} in the numerical examples.}
\label{HL-table-DA-designs}
\end{table}

If we let $m_j := |X_j|,$ the cardinality of $X_j,$ and let $t_j$ be 
the strength of $X_j$,
then the extremal spherical designs listed in Table~\ref{HL-table-DA-designs} have $m_j=(t_j+1)^2,$
and the low cardinality spherical designs have $m_j=(t_j+1)^2/2$ or $m_j=(t_j+1)^2/2 + 1$,
and in all cases $t_j \geqslant \sqrt{m_j}-1.$
It is not yet known if a sequence of spherical designs satisfying this lower bound on strength
can be extended indefinitely, 
but there is rigorous computational proof for $t_j$ up to 100 \cite{CheFL11}.
Also, it is now known that an infinite sequence of spherical designs exists 
with the required asymptotic order of strength,
that is, there is a sequence of spherical designs of 
cardinality $m_j = \mathcal O(t_j^2),$ \cite{BonRV13}
but the proof is not constructive and the corresponding implied constant is still unknown. 

One difference between the constructions for $\Sphere^1$ and $\Sphere^2$ 
is that the nesting of spherical designs is not efficient.
The union of two spherical designs of strengths $t_1$ and $t_2$ is in general, 
a spherical design whose strength
is the \emph{minimum} of $t_1$ and $t_2$.
For the approximate spherical designs listed in Table~\ref{HL-table-DA-designs},
the first design of strength 0 is a single point. 
The next design of strength 1 consists of two antipodal points, so nesting is efficient in this case.
After this, the resulting unions of spherical designs, in general, have strength no greater than 1.
This is one motivation for using optimal weight rules rather than equal weight rules on
a single sphere. 
In the optimal weight case, adding points to a rule does not increase its worst case error.

We now turn to estimates for rules on a single sphere, in order to use them with
Theorem~\ref{HL-theorem-new-3}.
On a single unit sphere, the \textsc{da} sparse grid quadrature rule whose worst case error is to be estimated 
is an optimal weight rule 
$q_j = q_r^{\gamma}(S_j),$ 
based on the increasing union of the approximate spherical designs from Table~\ref{HL-table-DA-designs},
$S_j := \bigcup_{i=0}^j X_i$.
Its worst case error is therefore no larger than that of the 
optimal weight rule $q_r^{\gamma}(X_j)$ based on $X_j,$
the largest spherical design contained in the union,
which is, in turn no greater than the worst case error of the 
equal weight rule $q_r^{\gamma,(\QMC)}(X_j)$ based on $X_j,$
with weights $1/|X_j|$,
\begin{align*}
e^2(q_j) 
&<  
e^2\big(q_r^{\gamma}(X_j)\big) \leqslant  
e^2\big(q_r^{\gamma,(\QMC)}(X_j)\big).
\end{align*}

According to Hesse, Kuo and Sloan \cite[Theorem 4]{HesKS07}, for $\gamma=1$, we have the bound
\begin{align*}
e^2\big(q_r^{1,(\QMC)}(X_j)\big) 
&\leqslant c t_j^{-2 r},
\intertext{and therefore}
e^2\big(q_r^1(X_j)\big) 
&\leqslant c (\sqrt{m_j}-1)^{-2 r} 
\leqslant C_1 m_j^{-r},
\end{align*}
for some $C_1>0.$
For general $\gamma,$ as per Kuo and Sloan \cite{KuoS05}, we have
\begin{align}
e^2\big(q_r^{\gamma,(\QMC)}(X_j)\big)
&= 
-1 + \frac{1}{m_j^2} \sum_{h=1}^{m_j} \sum_{i=1}^{m_j} \kernel_{1,\gamma}^{(r)}(x_{j,h},x_{j,i})
\notag
\\
&=
\gamma \frac{1}{m_j^2} \sum_{h=1}^{m_j} \sum_{i=1}^{m_j} A_r(x_{j,h} \cdot x_{j,i})
\leqslant \gamma C_1 m_j^{-r}.
\label{HL-eq-pre-new-39}
\end{align}
For the sequence of spherical designs given by Table~\ref{HL-table-DA-designs}, we also have
\begin{align}
2^j &\leqslant m_j \leqslant 2^j+1.
\notag
\intertext{We therefore have}
e^2(q_j) &\leqslant \gamma C_2\ 2^{-r j},
\label{HL-eq-e2-ineq}
\end{align}
for some $C_2>0.$
Recall that 
\begin{align*}
e^2(q_j) &= 
1 - \norm{q_j}_{\Hilbert_{1,\gamma}^{(r)}}^2 
\\
&=
1  - \norm{q_{j-1}}_{\Hilbert_{1,\gamma}^{(r)}}^2 - \norm{q_j-q_{j-1}}_{\Hilbert_{1,\gamma}^{(r)}}^2
\\
&=
e^2(q_{j-1}) - \norm{q_j-q_{j-1}}_{\Hilbert_{1,\gamma}^{(r)}}^2,
\end{align*}
for $j \geqslant 1$,
since $q_j-q_{j-1}$ is orthogonal to $q_{j-1}$.
Therefore
\begin{align*}
e^2(q_{j-1}) &= e^2(q_j) + \norm{q_j-q_{j-1}}_{\Hilbert_{1,\gamma}^{(r)}}^2.
\end{align*}
Since $e^2(q_j) \geqslant 0$, using~\eqref{HL-eq-e2-ineq} we obtain
$\norm{q_j-q_{j-1}}_{\Hilbert_{1,\gamma}^{(r)}}^2 \leqslant \gamma C_2\ 2^{-r (j-1)}.$
This, in turn implies that
\begin{align*}
\norm{q_j-q_{j-1}}_{\Hilbert_{1,\gamma}^{(r)}}^2 & \leqslant \gamma C_3\ 2^{-r j},
\end{align*}
where $C_3 = 2^{r} C_2.$

All the approximate spherical designs listed in Table~\ref{HL-table-DA-designs}
have one point in common, the ``north pole'' $(0,0,1).$
Therefore the number of points $n_j = |S_j|$ of the \textsc{da} quadrature rule 
with this sequence of point sets satisfies 
$n_j = 1 + \sum_{i=0}^j (m_i - 1) = -j + \sum_{i=0}^j m_i.$
Since $2^i \leqslant m_i \leq 2^i+1,$ we have $2^{j+1}-j-1 \leqslant n_j \leqslant 2^{j+1},$
and so $m_j \leqslant n_j \leqslant m_{j+1},$ for $j \geqslant 0.$
In view of~\eqref{HL-eq-pre-new-39}, and the preceding argument,
criteria~\eqref{HL-eq-new-39} and~\eqref{HL-eq-new-36} 
hold with $D=2^{-r/2},$ $C=\sqrt{C_3}$ and $\rho=2/r.$

Consider now that the spherical design for $j=0$ consists of only one point, the north pole,
and the spherical design for $j=1$ consists of two points, the north pole and its antipode,
the ``south pole'' $(0,0,-1).$
From \eqref{HL-eq-pre-new-39} we have
\begin{align*}
e^2\big(q_r^{\gamma,(\QMC)}(X_0)\big)
&=
\gamma A_r(1),~\text{and}
\\
e^2\big(q_r^{\gamma,(\QMC)}(X_1)\big)
&=
\frac{\gamma}{4} \sum_{h=1}^{2} \sum_{i=1}^{2} A_r(x_{1,h} \cdot x_{1,i})
=
\frac{\gamma}{2} \big( A_r(1) + A_r(-1) \big),~\text{so that}
\\
\norm{q_j-q_{j-1}}_{\Hilbert_{1,\gamma}^{(r)}}^2
&=
e^2\big(q_r^{\gamma,(\QMC)}(X_0)\big) - e^2\big(q_r^{\gamma,(\QMC)}(X_1)\big)
\\
&=
\frac{\gamma}{2} \big( A_r(1) - A_r(-1) \big)
\leqslant \gamma C_3\ 2^{-r},
\end{align*}
yielding $C_3 \geqslant 2^{r-1} \big( A_r(1) - A_r(-1) \big).$
In other words,
\begin{align*}
C
&\geqslant
2^{(r-1)/2} \sqrt{ A_r(1) - A_r(-1) }.
\end{align*}

Note that the sequence of point sets used on a single sphere,
as listed in Table \ref{HL-table-DA-designs} does not vary with $\gamma,$
as would be the case in the original setting of Wasil\-kowski and Wo\'znia\-kowski~\cite{WasW99}.
The estimates above show that in the setting of this paper it is not necessary to vary the sequence in this way.

The numerical examples of this Section use exponentially decreasing dimension weights,
partly because this guarantees strong tractability \cite[Theorem 4]{KuoS05},
and partly because this allows comparison with the numerical examples
used by Hesse, Kuo and Sloan in the \textsc{qmc} case \cite[Section 7]{HesKS07}.
To see how the rules~$q^{(\DA)}$ and~$q^{(\WW)}$ behave as the decay of the dimension weights is varied,
the numerical examples used here use $r=3$ with $\gamma_k=g^k$, for $g=0.1$, $0.5$, and~$0.9$,

For the $q^{(\WW)}$ rules, we use the definition~\eqref{HL-def-q-N-WW}
with $\xi_{d,k}:= C D,$ where $D=2^{-r/2}$ as above, and 
\begin{align*}
C &:= 2^{(r-1)/2} \sqrt{ A_r(1) - A_r(-1) } = \sqrt{ A_3(1) - A_3(-1) }  \simeq 1.7426.
\end{align*}
To further validate these values,
we use the~$q^{(\DA)}$ rules with $r=3,$ $d=1$ and $\gamma_1=1,$ to verify numerically that
\eqref{HL-eq-new-39} holds for all of the sets $S_j$ generated from the designs in
Table~\ref{HL-table-DA-designs}.
In fact, numerically, \eqref{HL-eq-new-39} also holds for $D$ as above and
$C =\norm{q_1-q_0}_{\Hilbert_{1,1}^{(3)}} = \sqrt{e^2(q_0)-e^2(q_1)} \simeq 1.453.$
That is,
\begin{align*}
2^{r j / 2} \norm{q_j-q_{j-1}}_{\Hilbert_{1,1}^{(3)}} 
&\leqslant
\norm{q_1-q_0}_{\Hilbert_{1,1}^{(3)}},~\text{for}~j=2 \ldots 11.
\end{align*}

Each of the numerical examples described below uses 
a particular dimension~$d$, from $d=1$ to~$16$; 
a particular maximum $1$-norm for indices, typically~$20$; and a particular maximum number of
points, up to~$100\,000$.

The numerical results are potentially affected by three problems.
First, if $\gamma$~is close to zero, and the number of points is large,
then the matrix used to compute the weights becomes ill-conditioned, and the weights may become inaccurate.
In this case, a least squares solution is used to obtain a best approximation to the weights.
Second, if the current squared error is close to zero,
and the squared norm for the current index is close to machine epsilon, then severe cancellation may occur.
Third, the sequence of spherical designs used in our numerical examples is finite,
so it is quite possible that our algorithm generates an index corresponding to
a spherical design which is not included in our finite set.
In these last two cases, the calculation of the quadrature rule is terminated.

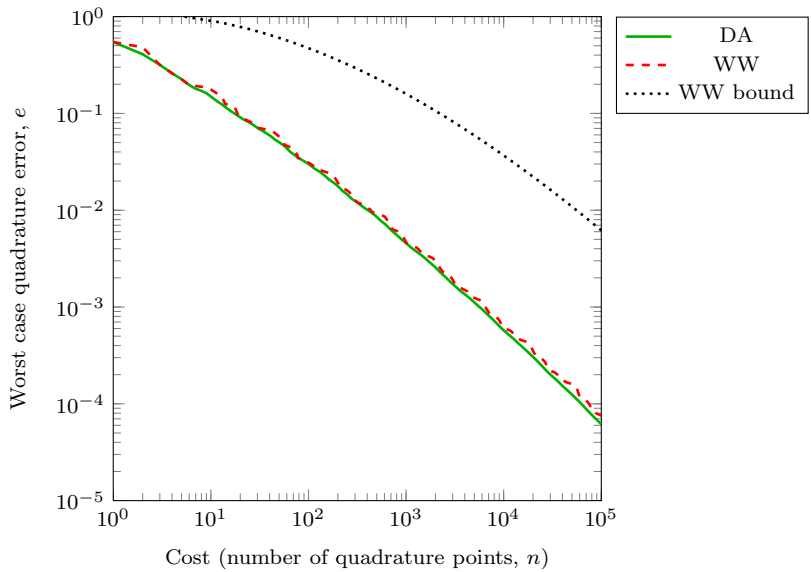
\begin{figure}[!ht]
\centering
\begin{tikzpicture}
\begin{loglogaxis}[
height=80mm,
width=80mm,
legend pos=outer north east,
xmin=1.0,
xmax=1.0e5,
ymin=1.0e-5,
ymax=1.0,
xlabel={Cost (number of quadrature points, $n$)},
ylabel={Worst case quadrature error, $e$}
]
\addplot[line width=1pt, color=green!70!black, mark=none] table[x=Cost,y=Error] {serrwtpcull-3-04-0.5-clean.dat};
\addplot[line width=1pt, color=red,   dashed,  mark=none] table[x=Cost,y=Error] {serrwtpwwcount-3-04-0.5-corrected-2.dat};
\addplot[line width=1pt, color=black, dotted,  mark=none] table[x=Cost,y=Error] {wwminlogcost-3-04-0.5-corrected.dat};
\legend{DA,WW,WW bound}
\end{loglogaxis}
\end{tikzpicture}
\caption{Error of \textsc{da} and \textsc{ww} rules vs \textsc{ww} bound for~$(\Sphere^2)^4$, $r=3$\,, $\gamma_{4,k}=0.5^k$.}
\label{HL-figure-0-5-HL-vs-WW}
\end{figure}

Figure~\ref{HL-figure-0-5-HL-vs-WW} 
is a log-log plot displaying the typical convergence behaviour of the \textsc{da} and \textsc{ww} rules
for the cases examined.
The case shown is that of~$(\Sphere^2)^4$, $r=3$\,, $\gamma_{4,k}=0.5^k$.
The number of points used varies from $n=1$ to~$100\,000$.
The cost axis is horizontal and the error axis is vertical, 
to match the figures shown in the torus paper \cite{HegL11}.
The curve in Figure~\ref{HL-figure-0-5-HL-vs-WW} labelled 
``\textsc{ww} bound'' is the minimum
of the bounds given by Theorem~\ref{HL-theorem-new-3}, 
as the parameter $\eta$ is varied over a finite number of values between 0 and 1.

In general, the \textsc{da} algorithm has a cost no greater than that of the \textsc{ww} algorithm.
Both are bounded by the \textsc{ww} bound of Theorem~\ref{HL-theorem-new-3}.
The \textsc{ww} cost bound itself has an asymptotic rate of convergence of 
$\mathcal O(\varepsilon^{-\rho}) = \mathcal O(\varepsilon^{-2/r}) = \mathcal O(\varepsilon^{-2/3})$ for all of our cases.
In other words, the asymptotic bound has quadrature error of order $\mathcal O(n^{-r/2}) = \mathcal O(n^{-3/2})$.

Judging from the slopes of the curves in Figure~\ref{HL-figure-0-5-HL-vs-WW}, 
the rates of convergence of both algorithms appear consistent with that of the bound,
but the asymptotic rate is not achieved by either algorithm or by the bound itself,
for the number of quadrature points displayed in the plot.

\begin{figure}[!ht]
\centering
\begin{tikzpicture}
\begin{loglogaxis}[
height=80mm,
width=80mm,
legend pos=outer north east,
xmin=1.0,
xmax=1.0e5,
ymin=1.0e-5,
ymax=1.0,
xlabel={Cost (number of quadrature points, $n$)},
ylabel={Worst case quadrature error, $e$}
]
\addplot[line width=1pt, color=black,          mark=+, mark size={2pt}] table[x=Cost,y=Error] {serrwtpcull-3-01-0.1-clean.dat};
\addplot[line width=1pt, color=brown!80!black, mark=x, mark size={2pt}] table[x=Cost,y=Error] {serrwtpcull-3-02-0.1-clean.dat};
\addplot[line width=1pt, color=blue,           mark=*, mark size={1pt}] table[x=Cost,y=Error] {serrwtpcull-3-04-0.1-clean.dat};
\addplot[line width=1pt, color=green!70!black, mark=none]               table[x=Cost,y=Error] {serrwtpcull-3-08-0.1-clean.dat};
\addplot[line width=1pt, color=red,    dashed, mark=none]               table[x=Cost,y=Error] {serrwtpcull-3-16-0.1-clean.dat};
\legend{$d=1$, $d=2$, $d=4$, $d=8$, $d=16$}
\end{loglogaxis}
\end{tikzpicture}
\caption{Error of \textsc{da} rules for~$(\Sphere^2)^d$, $d= 1,2,4,8,16$\,; $r=3$\,, $\gamma_{d,k}=0.1^k$.}
\label{HL-figure-0-1-1-2-4-8-16}
\end{figure}
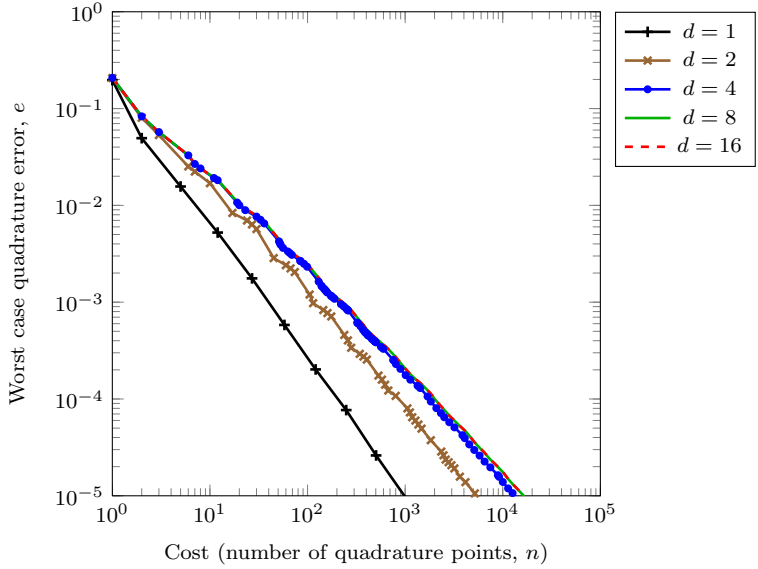

For $\gamma_{d,k}=0.1^k$, Figure~\ref{HL-figure-0-1-1-2-4-8-16} shows how the convergence rate of the error of
the \textsc{da} quadrature rules varies with dimension~$d$, for $d=1$\,, $2$, $4$, $8$, and~$16$.
The curve for $d=1$ appears consistent with the asymptotic error rate $\varepsilon=\mathcal O(N^{-3/2})$.
The cases $d=8$ and $d=16$ are almost indistinguishable on this Figure.
This is an example of the convergence in dimension.

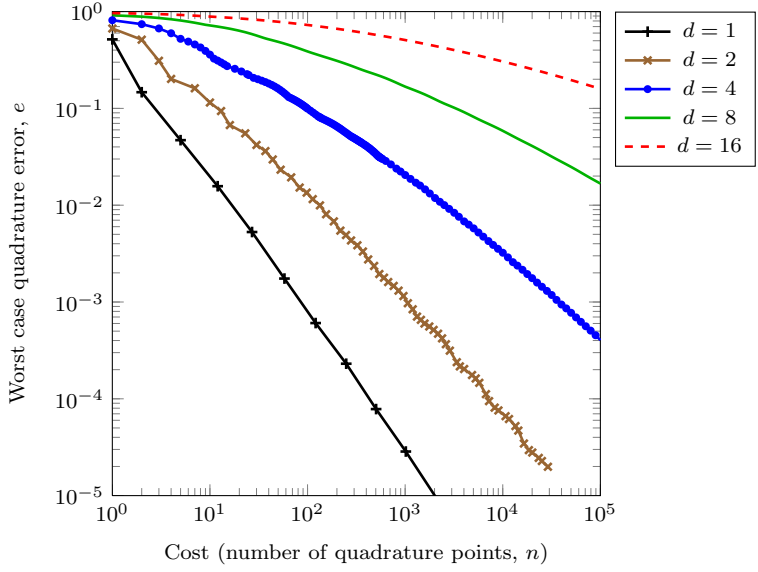
\begin{figure}[!ht]
\centering
\begin{tikzpicture}
\begin{loglogaxis}[
height=80mm,
width=80mm,
legend pos=outer north east,
xmin=1.0,
xmax=1.0e5,
ymin=1.0e-5,
ymax=1.0,
xlabel={Cost (number of quadrature points, $n$)},
ylabel={Worst case quadrature error, $e$}
]
\addplot[line width=1pt, color=black,          mark=+, mark size={2pt}] table[x=Cost,y=Error] {serrwtpcull-3-01-0.9-clean.dat};
\addplot[line width=1pt, color=brown!80!black, mark=x, mark size={2pt}] table[x=Cost,y=Error] {serrwtpcull-3-02-0.9-clean.dat};
\addplot[line width=1pt, color=blue,           mark=*, mark size={1pt}] table[x=Cost,y=Error] {serrwtpcull-3-04-0.9-clean.dat};
\addplot[line width=1pt, color=green!70!black, mark=none]               table[x=Cost,y=Error] {serrwtpcull-3-08-0.9-clean.dat};
\addplot[line width=1pt, color=red,    dashed, mark=none]               table[x=Cost,y=Error] {serrwtpcull-3-16-0.9-clean.dat};
\legend{$d=1$, $d=2$, $d=4$, $d=8$, $d=16$}
\end{loglogaxis}
\end{tikzpicture}
\caption{Error of \textsc{da} rules for~$(\Sphere^2)^d$, $d= 1,2,4,8,16$\,; $r=3$\,, $\gamma_{d,k}=0.9^k$.}
\label{HL-figure-0-9-1-2-4-8-16}
\end{figure}

Figure~\ref{HL-figure-0-9-1-2-4-8-16} shows the equivalent results for the \textsc{da} quadrature rules 
for $\gamma_{d,k}=0.9^k$.
The curve for $d=1$ again appears consistent with the asymptotic error rate $\varepsilon=\mathcal O(N^{-3/2})$,
but as $d$~increases to~$16$, the initial rate of convergence to zero of the error becomes much slower
than that for  $\gamma_{d,k}=0.1^k$.
For example, for $d=8,$ more than $1000$ quadrature points are needed to reduce the quadrature
error from 1 to 0.1, and for $d=16,$ more than $100\,000$ quadrature points are needed.
A similar phenomenon is observed for sparse grid quadrature with \textsc{rkhs} on the torus \cite{HegL11}.
\begin{appendices}
\section{Proof of Theorem \ref{HL-theorem-new-3}}
Assume that a sequence of quadrature points $x_1, x_2, \ldots \in \Sphere^2$ and 
a sequence of positive integers
$1=n_0 < n_1 < \ldots$ are given such that for all $\gamma \in (0,1]$
the corresponding optimal weight quadrature rules
\begin{align*}
q_j^{\gamma} &:= \sum_{\ell=1}^{n_j} w_{\ell}^{\gamma} k_{x_{\ell}}^{\gamma}  \in \Hilbert_{1,\gamma}^{(r)} 
\quad \text{satisfy} 
\end{align*}
\begin{align}
\norm{q_0^{\gamma}}_{\Hilbert_{1,\gamma}^{(r)}} 
&\leqslant 1, \quad
\norm{q_j^{\gamma} - q_{j-1}^{\gamma}}_{\Hilbert_{1,\gamma}^{(r)}}
\leqslant
\sqrt{\gamma} C D^j, \quad \text{and}
\label{HL-eq-proof-10}
\\  
\nu_j D^{j \rho} & \leqslant 1, \quad \text{for all~} j \geqslant 1,
\label{HL-eq-proof-11}
\end{align}
for some $D \in (0,1)$ and some positive $C$ and $\rho,$
where $\nu_0 := 1$ and $\nu_j := n_j-n_{j-1}$ for all $j \geqslant 1$.

Note that while the points $x_{\ell}$ do not change with $\gamma$,
the weights $w_{\ell}^{\gamma}$ \emph{do} change,
and so each rule $q_j^{\gamma}$ depends on $\gamma$ in general.

Define
\begin{align*}
b_j^{\gamma}
&:=
\begin{cases}
1, & j = 0,
\\
\sqrt{\gamma} C D^j, & j \geqslant 1;
\end{cases}
\quad
\delta_j^{\gamma}
:=
\begin{cases}
q_0^{\gamma}, & j = 0,
\\
q_j^{\gamma} - q_{j-1}^{\gamma}, & j \geqslant 1. 
\end{cases}
\end{align*}

Then~\eqref{HL-eq-proof-10} implies that
\begin{align*}
\norm{ \delta_j^{\gamma} }_{\Hilbert_{1,\gamma}^{(r)}}
&\leqslant
b_j^{\gamma} \quad \text{for all~} j \geqslant 0.
\end{align*}
As a consequence, and since $\norm{q_0^{\gamma}}_{\Hilbert_{1,\gamma}^{(r)}} \leqslant 1,$
it holds that
\begin{align}
\norm{ \Delta_j^{\gamma} }_{\Hilbert_{d,\gamma}^{(r)}}
&\leqslant
b(d,j),
\label{proof-11-1}
\end{align}
where
\begin{align*}
b(d,j)
&:=
\prod_{k=1}^d b_{j_k}^{\gamma_{d,k}}
=
\prod_{k=1}^d \left( \sqrt{\gamma_{d,k}}\ C D^{j_k} \right)^{1-\partial_{0,j_k}},
\end{align*}
as per~\eqref{HL-eq-b-d-j-def}.

Let $(\xi_{d,k}),$ $k=1,\ldots,d,$ be a sequence of positive numbers, 
with $\xi_{d,1} \geqslant \sqrt{1-D^2},$
and define
\begin{align*}
\xi(d,j) &:= \prod_{k=1}^d \xi_{d,k}^{1-\partial_{0,j_k}}. 
\end{align*}
as per~\eqref{HL-eq-xi-d-j-def}.
We intend to sort the incremental rules in order of decreasing $b(d,j)/\xi(d,j).$
For this order to agree with the lattice ordering,
it must be the case that $b_1^{\gamma_{d,k}}/\xi_{d,k} \leqslant 1,$
that is, $\xi_{d,k} \geqslant \sqrt{\gamma_{d,k}} C D$ for all $k.$

Let
\begin{align*}
I(d,\varepsilon,\eta) &:= 
\left\{j \relmiddle| b(d,j)/\xi(d,j) > \big( \varepsilon/C_1(d,\eta) \big)^{1/(1-\eta)} \right\},
\intertext{where $\eta \in (0,1)$ and}
C_1(d,\eta) 
&:= 
\sqrt{\frac{\xi_{d,1}^{2(1-\eta)}}{1-D^2} \prod_{k=2}^d \left(1 + (C^2 \gamma_{d,k})^{\eta}\ \xi_{d,k}^{2(1-\eta)} \frac{D^{2\eta}}{1-D^{2\eta}}\right)},
\end{align*}
as per~\eqref{HL-eq-C-1-d-eta-def}.

Let $N(d,\varepsilon,\eta) := \abs{I(d,\varepsilon,\eta)},$
and let $j^{(\WW)} = (j^{(\WW)}_{(1)},j^{(\WW)}_{(2)},\ldots)$ 
be an ordering of indices the indices in $\Indices$ by decreasing $b(d,j)/\xi(d,j).$
Then
\begin{align*}
I^{(\WW)}_{N(d,\varepsilon,\eta)}
&:=
\left\{j^{(\WW)}_{(1)},j^{(\WW)}_{(2)},\ldots,j^{(\WW)}_{(N(d,\varepsilon,\eta))}\right\}
=
I(d,\varepsilon,\eta).
\end{align*}
Define
\begin{align}
q^{(\WW)}_{d,\varepsilon,\eta}
&:=
\begin{cases}
\sum_{j \in I(d,\varepsilon,\eta)} \Delta_j, 
&\varepsilon < 1,
\\
0, 
&\varepsilon \geqslant 1,
\end{cases}
\label{HL-eq-proof-45}
\end{align}
in agreement with~\eqref{HL-eq-new-45} for $\varepsilon < 1.$

The proof now proceeds in two main steps, closely following Wasil\-kowski and Wo\'znia\-kowski's
proof of their \cite{WasW99} Theorem 3, for comparison purposes.
Firstly, we prove the error bound $e(q^{(\WW)}_{d,\varepsilon,\eta}) \leqslant \varepsilon.$
Secondly, we prove the cost bound~\eqref{HL-eq-WW-cost-bound}.

For $\varepsilon \geqslant 1,$ from~\eqref{HL-eq-e-def} and~\eqref{HL-eq-proof-45} 
we have $e(q^{(\WW)}_{d,\varepsilon,\eta})=1.$
Let 
\begin{align}
\alpha(d,\varepsilon,\eta) 
&:= \big( \varepsilon/C_1(d,\eta) \big)^{1/(1-\eta)}
\label{HL-eq-alpha-def}
\end{align}
so that
$I(d,\varepsilon,\eta) = \{j \mid b(d,j) > \xi(d,j)\ \alpha(d,\varepsilon,\eta) \}.$
From~\eqref{proof-11-1}, and using the orthogonality of the rules $\Delta_j,$ we therefore have
\begin{align*}
e^2(q^{(\WW)}_{d,\varepsilon,\eta})
&\leqslant
e^2(d,\varepsilon,\eta) :=
\begin{cases}
\sum_{b(d,j) \leqslant \xi(d,j)\ \alpha(d,\varepsilon,\eta)} b^2(d,j), 
&\varepsilon < 1,
\\
0, 
&\varepsilon \geqslant 1.
\end{cases}
\end{align*}

For $h=1,2,\ldots,d,$ we also define
\begin{align*}
\overline{e}^2(h,\varepsilon,\eta)
&:=
\begin{cases}
\sum_{\overline{b}(h,j) \leqslant \overline{\xi}(h,j)\ \overline{\alpha}(h,\varepsilon,\eta)} \overline{b}^2(h,j), 
&\varepsilon < 1,
\\
0, 
&\varepsilon \geqslant 1,
\end{cases}
\end{align*}
where
\begin{align}
\overline{b}(h,j)
&:=
\prod_{k=1}^h b_{j_k}^{\gamma_{d,k}}
=
\prod_{k=1}^h \left( \sqrt{\gamma_{d,k}}\ C D^{j_k} \right)^{1-\partial_{0,j_k}},
\notag
\\
\overline{\xi}(h,j) 
&:= 
\prod_{k=1}^h \xi_{d,k}^{1-\partial_{0,j_k}},
\quad
\overline{\alpha}(h,\varepsilon,\eta) 
:= \big( \varepsilon/\overline{C_1}(h,\eta) \big)^{1/(1-\eta)},
\label{HL-eq-ov-alpha-def}
\\
\overline{C_1}(h,\eta) 
&:= 
\sqrt{\frac{\xi_{d,1}^{2(1-\eta)}}{1-D^2} \prod_{k=2}^h \left(1 + (C^2 \gamma_{d,k})^{\eta}\ \xi_{d,k}^{2(1-\eta)} \frac{D^{2\eta}}{1-D^{2\eta}}\right)}.
\notag
\end{align}
Note that $\overline{e}(d,\varepsilon,\eta) = e(d,\varepsilon,\eta).$

We show by induction on $h$ that $\overline{e}(h,\varepsilon,\eta) \leqslant \varepsilon$ for all
$\varepsilon > 0.$
For $h=1,$ if $\varepsilon \geqslant 1,$ this is trivially true, since $\overline{e}(1,\varepsilon,\eta)=1.$
For $h=1$ and $\varepsilon < 1,$ let
\begin{align}
N^{*}_{\varepsilon}
&:=
\abs{ \{j \mid \overline{b}(1,j) > \overline{\xi}(1,j)\ \overline{\alpha}(1,\varepsilon,\eta) \} }.
\label{HL-eq-N-ast-def}
\end{align}
Since $\overline{b}(1,0) = 1,$ $\overline{\xi}(1,0) = 1,$ and
$\overline{\alpha}(1,\varepsilon,\eta) 
= \big( \varepsilon/\overline{C_1}(1,\eta) \big)^{1/(1-\eta)},$
it holds that if $\overline{C_1}(1,\eta) \geqslant 1,$ then $N^{*}_{\varepsilon} \geqslant 1.$
But
\begin{align*}
\overline{C_1}(1,\eta) 
&= 
\frac{\xi_{d,1}^{1-\eta}}{\sqrt{1-D^2}},
\end{align*}
and we have specified that $\xi_{d,1} \geqslant \sqrt{1-D^2},$ 
so it follows that $\overline{C_1}(1,\eta) \geqslant 1.$
Now
\begin{align*}
\overline{e}^2(1,\varepsilon,\eta) 
&=
\sum_{j \geqslant N^{*}_{\varepsilon}} \overline{b}^2(1,j)
=
\sum_{j \geqslant N^{*}_{\varepsilon}} \gamma_{d,1}\ C^2 D^{2 j}
\\
&=
\gamma_{d,1}\ C^2 D^{2 N^{*}_{\varepsilon}}\  \sum_{j \geqslant 0} D^{2 j}
=
\frac{\gamma_{d,1}\ C^2 D^{2 N^{*}_{\varepsilon}}}{1-D^2}
=
\frac{\overline{b}^2(1,N^{*}_{\varepsilon})}{1-D^2},
\intertext{but}
\overline{b}^2(1,N^{*}_{\varepsilon})
&\leqslant
\overline{\xi}^2(1,N^{*}_{\varepsilon})\ \overline{\alpha}^2(1,\varepsilon,\eta)
=\xi_{d,1}^2 \left( \frac{\varepsilon\, \sqrt{1-D^2}}{\xi_{d,1}^{1-\eta}} \right)^{2/(1-\eta)}
\\
&=
(\varepsilon\, \sqrt{1-D^2})^{2/(1-\eta)}
=
\varepsilon^{2/(1-\eta)}\ (1-D^2)^{1/(1-\eta)},
\intertext{so}
\overline{e}^2(1,\varepsilon,\eta) 
&=
\frac{\overline{b}^2(1,N^{*}_{\varepsilon})}{1-D^2}
\leqslant
\varepsilon^{2/(1-\eta)}\ (1-D^2)^{\eta/(1-\eta)}
\leqslant
\varepsilon^2.
\end{align*}
For the inductive step, let $1 < h \leqslant d$
and assume that $\overline{e}(h-1,\varepsilon,\eta) \leqslant \varepsilon.$
The following estimates show that $\overline{e}(h,\varepsilon,\eta) \leqslant \varepsilon.$

Noting that, from their definitions~\eqref{HL-eq-ov-alpha-def}, 
$\overline{b}(h,j) = \overline{b}(h-1,j)\ b_{j_h}^{\gamma_{d,h}}$
and $\overline{\xi}(h,j) = \overline{\xi}(h-1,j)\ \xi_{d,h}^{1-\partial_{0,j_h}},$
let
\begin{align}
\overline{\beta}_{\ell}(h,\varepsilon,\eta)
&:=
\frac{\xi_{d,h}^{1-\partial_{0,j_h}}\ \overline{\alpha}(h,\varepsilon,\eta)}{b_{\ell}^{\gamma_{d,h}}}, 
\label{HL-eq-ov-beta-def}
\end{align}
with $\overline{\alpha}$ as per~\eqref{HL-eq-ov-alpha-def}.
Since $\overline{b}(h,j) = \overline{b}(h-1,j)\ b_{\ell}^{\gamma_{d,h}}$ if and only if $\ell = j_h,$
and since 
\begin{align*}
\overline{\xi}(h,j)\ \overline{\alpha}(h,\varepsilon,\eta)
&=
\overline{\xi}(h-1,j)\ b_{\ell}^{\gamma_{d,h}}\ \overline{\beta}_{\ell}(h,\varepsilon,\eta),  
\end{align*}
we see that for $\varepsilon < 1,$
\begin{align}
\overline{e}^2(h,\varepsilon,\eta)
=
\sum_{\ell=0}^{\infty} \quad
\sum_{\overline{b}(h-1,j) \leqslant \overline{\xi}(h-1,j)\ \overline{\beta}_{\ell}(h,\varepsilon,\eta)} 
\overline{b}^2(h-1,j)\ (b_{\ell}^{\gamma_{d,h}})^2. 
\label{HL-eq-ov-e-2-ov-beta}
\end{align}

Now let
\begin{align}
a_{h,\eta}
&:=
\frac{1}{\sqrt{1 + (C^2 \gamma_{d,h})^{\eta}\ \xi_{d,h}^{2(1-\eta)} \frac{D^{2\eta}}{1-D^{2\eta}}}}.
\label{HL-eq-a-def}
\end{align}
Then $\overline{C_1}(h,\eta) = \overline{C_1}(h-1,\eta)/a_{h,\eta},$ so that
\begin{align*}
\overline{\alpha}(h,\varepsilon,\eta)
&=
\left( \frac{a_{h,\eta}\, \varepsilon}{\overline{C_1}(h-1,\eta)} \right)^{1/(1-\eta)}
=
\overline{\alpha}(h-1,a_{h,\eta}\, \varepsilon,\eta),
\end{align*}
and so
\begin{align}
\overline{\beta}_{\ell}(h,\varepsilon,\eta)
&=
\frac{\xi_{d,h}^{1-\partial_{0,j_h}}\ \overline{\alpha}(h,\varepsilon,\eta)}{b_{\ell}^{\gamma_{d,h}}}
=
\overline{\alpha}\left(h-1,
a_{h,\eta}\, \varepsilon\, \left(\frac{\xi_{d,h}^{1-\partial_{0,j_h}}}{b_{\ell}^{\gamma_{d,h}}}\right)^{1-\eta},\eta \right).
\label{HL-eq-ov-beta-ov-alpha}
\end{align}
In particular,
$\overline{\beta}_0(h,\varepsilon,\eta) = \overline{\alpha}(h-1,a_{h,\eta}\, \varepsilon,\eta),$
and therefore
\begin{align*}
\sum_{\overline{b}(h-1,j) \leqslant \overline{\xi}(h-1,j)\ \overline{\beta}_0(h,\varepsilon,\eta)} 
\overline{b}^2(h-1,j)
&=
\sum_{\overline{b}(h-1,j) \leqslant \overline{\xi}(h-1,j)\ \overline{\alpha}(h-1,a_{h,\eta}\, \varepsilon,\eta)} 
\overline{b}^2(h-1,j)
\\
&=
\overline{e}^2(h-1,a_{h,\eta}\, \varepsilon,\eta). 
\end{align*}
This implies that
\begin{align*}
\overline{e}^2(h,\varepsilon,\eta)
&=
\\
\sum_{\ell=0}^{\infty}  (b_{\ell}^{\gamma_{d,h}})^2
&\ \sum_{\overline{b}(h-1,j) \leqslant \overline{\xi}(h-1,j)
\ \overline{\alpha}\left(h-1,
a_{h,\eta}\, \varepsilon\, \left(\frac{\xi_{d,h}^{1-\partial_{0,j_h}}}{b_{\ell}^{\gamma_{d,h}}}\right)^{1-\eta},\eta \right)
} 
\overline{b}^2(h-1,j)
\\
&=
\overline{e}^2(h-1,a_{h,\eta}\, \varepsilon,\eta) +
\sum_{\ell=1}^{\infty}  (b_{\ell}^{\gamma_{d,h}})^2
\ \overline{e}^2\left(h-1,a_{h,\eta}\, \varepsilon \left(\frac{\xi_{d,h}}{b_{\ell}^{\gamma_{d,h}}}\right)^{1-\eta},\eta\right)
\\
&\leqslant
(a_{h,\eta}\, \varepsilon )^2
\ \left(1 + 
\sum_{\ell=1}^{\infty}  (b_{\ell}^{\gamma_{d,h}})^2
\left(\frac{\xi_{d,h}}{b_{\ell}^{\gamma_{d,h}}}\right)^{2(1-\eta)} \right)
\\
&=
(a_{h,\eta}\, \varepsilon)^2
\ \left(1 + 
\xi_{d,h}^{2(1-\eta)}
\sum_{\ell=1}^{\infty}  (b_{\ell}^{\gamma_{d,h}})^{2\eta} \right)
\\
&=
(a_{h,\eta}\, \varepsilon)^2
\ \left(1 + 
\xi_{d,h}^{2(1-\eta)}
\gamma_{d,h}^{\eta}\ C^{2\eta} \frac{D^{2\eta}}{1-D^{2\eta}} \right)
=
\frac{(a_{h,\eta}\, \varepsilon)^2}{a_{h,\eta}^2}
=
\varepsilon^2.
\end{align*}

We now analyze the cost of $q^{(\WW)}_{d,\varepsilon,\eta}.$
As per Definition~\ref{HL-def-nu-p},
the cost of the incremental rule $\Delta_j = \bigotimes_{k=1}^d \delta_{j_k}^{(k)}$
is $\nu_j := \prod_{k=1}^d \nu_{j_k},$
and therefore
\begin{align*}
\operatorname{cost}&(q^{(\WW)}_{d,\varepsilon,\eta})
&=
c(d,\varepsilon,\eta)
:=
\begin{cases}
\sum_{b(d,j) \leqslant \xi(d,j)\ \alpha(d,\varepsilon,\eta)} \prod_{k=1}^d \nu_{j_k}, 
&\varepsilon < 1,
\\
0, &\varepsilon \geqslant 1,
\end{cases}
\end{align*}
with $\alpha(d,\varepsilon,\eta)$ as per~\eqref{HL-eq-alpha-def}.
For $h \in \{1,\ldots,d\},$ define
\begin{align*}
\overline{c}(h,\varepsilon,\eta)
&:=
\begin{cases}
\sum_{\overline{b}(h,j) \leqslant \overline{\xi}(h,j)\ \overline{\alpha}(h,\varepsilon,\eta)} \prod_{k=1}^h \nu_{j_k}, 
&\varepsilon < 1,
\\
0, &\varepsilon \geqslant 1,
\end{cases}
\end{align*}
with $\overline{b},$ $\overline{\xi}$ and $\overline{\alpha}$ as per~\eqref{HL-eq-ov-alpha-def}.
Clearly, $\overline{c}(d,\varepsilon,\eta) = c(d,\varepsilon,\eta).$

For $h=1,$ we use $N^{*}_{\varepsilon}$ as per~\eqref{HL-eq-N-ast-def}, so for $\varepsilon < 1$ we have
\begin{align*}
\overline{c}(1,\varepsilon,\eta) 
&=
\sum_{j=0}^{N^{*}_{\varepsilon}-1} \nu_j.
\end{align*}
The condition~\eqref{HL-eq-proof-11} then results in the bound
\begin{align}
\overline{c}(1,\varepsilon,\eta)
&\leqslant
\sum_{j=0}^{N^{*}_{\varepsilon}-1} D^{-j \rho}
=
\frac{D^{-\rho N^{*}_{\varepsilon}} - 1}{D^{-\rho} - 1}
=
D^{\rho} \frac{D^{-\rho N^{*}_{\varepsilon}} - 1}{1 - D^{\rho}}
\notag
\\
&\leqslant
\frac{D^{-\rho (N^{*}_{\varepsilon} - 1)}}{1 - D^{\rho}}.
\label{HL-eq-ov-c-1-bound-1}
\end{align}
From~\eqref{HL-eq-N-ast-def}, we have
\begin{align*}
N^{*}_{\varepsilon}
&=
\abs{ \{j \mid \overline{b}(1,j) > \overline{\xi}(1,j)\ \overline{\alpha}(1,\varepsilon,\eta) \} }
\\
&=
\abs{ \{j \mid \overline{b}(1,j) > \xi_{d,1}^{1-\partial_{0,j}}\ \overline{\alpha}(1,\varepsilon,\eta) \} }.
\end{align*}
Using~\eqref{HL-eq-ov-alpha-def},
we see that $N^{*}_{\varepsilon}$ is the smallest integer $j>0$ such that
\begin{align*}
\sqrt{\gamma_{d,1}}\ C D^j
&\leqslant
\xi_{d,1}\ \overline{\alpha}(1,\varepsilon,\eta).
\end{align*}
This, in turn, implies that either $N^{*}_{\varepsilon} = 1,$ or,
using the definitions of $\overline{\alpha}$ and $\overline{C_1}$ from~\eqref{HL-eq-ov-alpha-def},
\begin{align*}
D^{N^{*}_{\varepsilon} - 1}
&>
\frac{\xi_{d,1}}{\sqrt{\gamma_{d,1}}\ C}\ \overline{\alpha}(1,\varepsilon,\eta)
=
\frac{\xi_{d,1}}{\sqrt{\gamma_{d,1}}\ C}\ \big( \varepsilon/\overline{C_1}(1,\eta) \big)^{1/(1-\eta)}
\\
&=
\frac{(1-D^2)^{1/(2-2\eta)}}{\sqrt{\gamma_{d,1}}}\ \varepsilon^{1/(1-\eta)}.
\end{align*}
If $N^{*}_{\varepsilon} = 1,$ then since $\xi_{d,1} \geqslant \sqrt{1-D^2},$
it holds that
\begin{align*}
1 = D^{N^{*}_{\varepsilon} - 1}
&>
\overline{\alpha}(1,\varepsilon,\eta)
=
\frac{(1-D^2)^{1/(2-2\eta)}}{\xi_{d,1}}\ \varepsilon^{1/(1-\eta)}.
\end{align*}

Let
\begin{align*}
C_1^{*} 
&:= 
(1-D^2)^{1/(2-2\eta)}\ \operatorname{min}\left(
 \frac{1}{\sqrt{\gamma_{d,1}}}, \frac{1}{\xi_{d,1}} 
\right),
\end{align*}
so that
$D^{N^{*}_{\varepsilon} - 1} > C_1^{*} \varepsilon^{1/(1-\eta)}.$
The bound~\eqref{HL-eq-ov-c-1-bound-1} then becomes
\begin{align*}
\overline{c}(1,\varepsilon,\eta)
&\leqslant
\frac{D^{-\rho (N^{*}_{\varepsilon} - 1)}}{1 - D^{\rho}}
<
C_1 (1/\varepsilon)^{\rho/(1-\eta)}, 
\end{align*}
where
\begin{align}
C_1 
&:= \frac{(C_1^{*})^{-\rho}}{1-D^{\rho}}
=
\frac{\operatorname{max}(\sqrt{\gamma_{d,1}}, \xi_{d,1})^{\rho}}{(1 - D^{\rho})(1-D^2)^{\rho/(2-2\eta)}}.
\label{HL-eq-C-1-def}
\end{align}

For $h \geqslant 2$ assume that $\overline{c}(h-1,\varepsilon,\eta) \leqslant C_{h-1} (\rho/\varepsilon)^{1/(1-\eta)}$
for some $C_{h-1} > 0.$

An argument similar to that used to establish~\eqref{HL-eq-ov-e-2-ov-beta}
shows that
\begin{align*}
\overline{c}(h,\varepsilon,\eta)
&=
\sum_{\ell=0}^{\infty} \nu_{\ell}
 \sum_{\overline{b}(h-1,j) \leqslant \overline{\xi}(h-1,j)\ \overline{\beta}_{\ell}(h,\varepsilon,\eta)} 
  \prod_{k=1}^{h-1} \nu_{j_k}
\end{align*}
with $\overline{b}$ and $\overline{\xi}$ as per~\eqref{HL-eq-ov-alpha-def}
and $\overline{\beta}_{\ell}$ as per~\eqref{HL-eq-ov-beta-def}.
Using~\eqref{HL-eq-a-def} and~\eqref{HL-eq-ov-beta-ov-alpha}, we see that
\begin{align}
\overline{c}(h,\varepsilon,\eta)
&=
\overline{c}(h-1,a_{h,\eta}\,\varepsilon,\eta) +
\sum_{\ell=1}^{\overline{g}(h,\varepsilon,\eta)} 
\nu_{\ell}\ \overline{c}\left(
h-1,a_{h,\eta}\,\varepsilon\, \left(\frac{\xi_{d,h}}{C\sqrt{\gamma_{d,h}}\,D^{\ell}}\right)^{1-\eta},\eta
\right),
\label{HL-eq-ov-c-h-1}
\end{align}
where $\overline{g}(h,\varepsilon,\eta)$ is the largest integer $j$ such that
\begin{align*}
C\sqrt{\gamma_{d,h}}\,D^j
&>
\xi_{d,h}\ \overline{\alpha}(h,\varepsilon,\eta),
\end{align*}
and therefore
\begin{align*}
\overline{g}(h,\varepsilon,\eta) 
&= 
\left\lfloor
 \frac{ \log\left( 
  C \sqrt{\gamma_{d,h}}/\big(\xi_{d,h}\ \overline{\alpha}(h,\varepsilon,\eta)\big) 
 \right) }{\log (D^{-1})} 
\right\rfloor_{+}.
\end{align*}
From~\eqref{HL-eq-ov-alpha-def} we know that
\begin{align*}
1/\overline{\alpha}(h,\varepsilon,\eta) 
&= \overline{C_1}(h,\eta)^{1/(1-\eta)} \varepsilon^{-1/(1-\eta)}.
\end{align*}
From~\eqref{HL-eq-ov-alpha-def} and~\eqref{HL-eq-a-def}
it also follows that
\begin{align*}
\overline{C_1}(h,\eta)
&=
\frac{\xi_{d,1}^{1-\eta}}{\sqrt{1-D^2}\prod_{k=2}^h a_{k,\eta}}.
\end{align*}
From~\eqref{HL-eq-f-def} and~\eqref{HL-eq-a-def}, we see that $a_{h,\eta}^{-1/(1-\eta)} =f_{h,\eta},$
and therefore
\begin{align*}
1/\overline{\alpha}(h,\varepsilon,\eta) 
&= \xi_{d,1}\ (1-D^2)^{-1/(2-2\eta)} \left( \prod_{k=2}^h a_{k,\eta}^{-1/(1-\eta)} \right) \varepsilon^{-1/(1-\eta)}
\\
&= \xi_{d,1}\ (1-D^2)^{-1/(2-2\eta)} \left( \prod_{k=2}^h f_{k,\eta} \right) \varepsilon^{-1/(1-\eta)}.
\end{align*}
Therefore
\begin{align*}
\overline{g}(h,\varepsilon,\eta) 
&= 
\left\lfloor
 \frac{ \log\left( 
  \frac{\displaystyle C \gamma_{d,h}^{1/2}}{\displaystyle (1-D^2)^{1/(2-2\eta)}}\ 
  \frac{\displaystyle \xi_{d,1}}{\displaystyle \xi_{d,h}}
  \left( 
   \prod_{k=2}^h f_{k,\eta} 
  \right) 
  \left(\frac{\displaystyle 1}{\displaystyle \varepsilon}\right)^{1/(1-\eta)}
 \right) }{\log \big(D^{-1}\big)} 
\right\rfloor_{+}
\\
&= g(h,\varepsilon,\eta) \quad \text{as per~\eqref{HL-eq-g-def}.}
\end{align*}

From the inductive assumption and the sum~\eqref{HL-eq-ov-c-h-1}, it follows that
\begin{align*}
\overline{c}(h,\varepsilon,\eta)
&\leqslant
C_{h-1} (a_{h,\eta}\,\varepsilon)^{-\rho/(1-\eta)} 
\left(
1
+
\sum_{\ell=1}^{\overline{g}(h,\varepsilon,\eta)} 
\nu_{\ell}\ \left(\frac{C\sqrt{\gamma_{d,h}}\,D^{\ell}}{\xi_{d,h}}\right)^{\rho}
\right)
\\
&=
C_{h-1}(a_{h,\eta}\,\varepsilon)^{-\rho/(1-\eta)} 
\left(
 1 +
 \frac{C^{\rho} \gamma_{d,h}^{\rho/2}}{\xi_{d,h}^{\rho}}
 \sum_{\ell=1}^{g(h,\varepsilon,\eta)} 
 \nu_{\ell}\ D^{\ell\rho}
\right).
\end{align*}
%
Therefore the bound~\eqref{HL-eq-proof-11} implies the bound
\begin{align*}
\overline{c}(h,\varepsilon,\eta)
&\leqslant
C_{h-1}(a_{h,\eta}\,\varepsilon)^{-\rho/(1-\eta)} 
\left(
 1 +
 \frac{C^{\rho} \gamma_{d,h}^{\rho/2}}{\xi_{d,h}^{\rho}}\,g(h,\varepsilon,\eta)
\right)
=
C_h\ \left( \frac{1}{\varepsilon} \right)^{\rho/(1-\eta)},
\intertext{where (using~\eqref{HL-eq-C-1-def})}
C_h
&:=
C_{h-1}\ a_{h,\eta}^{-\rho/(1-\eta)} 
\left(
 1 +
 \frac{C^{\rho} \gamma_{d,h}^{\rho/2}}{\xi_{d,h}^{\rho}}\,g(h,\varepsilon,\eta)
\right)
\\
&=
C_{h-1}
\left(
 1 +
 \frac{C^{\rho} \gamma_{d,h}^{\rho/2}}{\xi_{d,h}^{\rho}}\,g(h,\varepsilon,\eta)
\right)
f_{h,\eta}^{\rho}.
\end{align*}

By induction, we therefore have
\begin{align*}
C_d
&=
C_1
\prod_{k=2}^d
\left(
 1 +
 \frac{C^{\rho} \gamma_{d,k}^{\rho/2}}{\xi_{d,k}^{\rho}}\, g(k,\varepsilon,\eta)
\right)
f_{k,\eta}^{\rho}
\\
&=
\frac{\operatorname{max}(\sqrt{\gamma_{d,1}}, \xi_{d,1})^{\rho}
{\displaystyle \prod_{k=2}^d}
\left(
 1 +
 \frac{\displaystyle C^{\rho} \gamma_{d,k}^{\rho/2}}{\displaystyle \xi_{d,k}^{\rho}}\ g(k,\varepsilon,\eta)
\right)
f_{k,\eta}^{\rho}
}{(1 - D^{\rho})(1-D^2)^{\rho/(2-2\eta)}}
\\
&=
C(d,\varepsilon,\eta), \quad \text{as per~\eqref{HL-eq-C-d-def}.} &\qed
\end{align*}
\end{appendices}

\section*{Acknowledgements}
Thanks to Rob Womersley for the spherical designs,
and to Gary Froyland for discussions on precedence-constrained knapsack problems.
The hospitality of the Hausdorff Research Institute for Mathematics (HIM) 
in Bonn is much appreciated.
The support of the Australian Research Council under its Centre of Excellence program 
is gratefully acknowledged.



\end{document}